\documentclass[12pt]{amsart}

\usepackage{amsmath,amsthm,amscd,euscript,longtable}
\usepackage{graphicx}
\usepackage{epstopdf}
\def \r{\mathbb R}

\def \z{\mathbb Z}

\def \({\langle}
\def \){\rangle}


\DeclareMathOperator{\modd}{mod}
\DeclareMathOperator{\vrm}{Vrm}
 
\DeclareMathOperator{\sign}{sign} 

\newtheorem{theorem}{Theorem}[section]
\newtheorem{lemma}[theorem]{Lemma}

\newtheorem{proposition}[theorem]{Proposition}
\newtheorem{corollary}[theorem]{Corollary}
\theoremstyle{remark}
\newtheorem{remark}[theorem]{Remark}
\theoremstyle{definition}
\newtheorem{definition}[theorem]{Definition}
\newtheorem{example}[theorem]{Example}
\newtheorem{problem}{Problem}
\newtheorem{conjecture}[problem]{Conjecture}

\title[Geometry and combinatoric of Minkowski--Voronoi complexes]
{Geometry and combinatoric of Minkowski--Voronoi 3-dimesional continued fractions}

\author{Oleg Karpenkov}
\address[Oleg Karpenkov]{Department of Mathematics,
Mathematical Sciences Building,
Peach Street, University of Liverpool,
Liverpool L69 7ZL, United Kingdom}
\email{karpenk@liv.ac.uk}

\author{Alexey Ustinov}
\address[Alexey Ustinov]
{Pacific National University
136 Tihookeanskaya st., Khabarovsk, 680035, Russia
AND
Institute of Applied Mathematics, Khabarovsk Division,
Far Eastern Branch of the Russian Academy of Sciences
54 Dzerzhinsky Street, Khabarovsk, 680000, Russia}
\email{ustinov@iam.khv.ru}

\date{1 February 2017}

\thanks
{
Oleg Karpenkov is partially supported by EPSRC grant EP/N014499/1 (LCMH)
}
\thanks{
Alexey Ustinov is supported by Grant of the Government of Khabarovsk Krai (Order N 479-p, June 29, 2016).
}


\usepackage{color}

\begin{document}
\input{epsf}

\maketitle

\begin{abstract}
In this paper we investigate the combinatorial structure of 3-dimensional Minkowski-Voronoi continued fractions.
Our main goal is to prove the asymptotic stability of Minkowski-Voronoi complexes in special two-parametric families
of rank-1 lattices.
In addition we construct explicitly the complexes for the case of White's rank-1 lattices
and provide with a hypothetic description in a more complicated settings.
\end{abstract}

\tableofcontents

\section{Introduction}

Consider a pair of positive integers $(a,N)$, and let $a<N$.

Then the lengths of ordinary continued fractions for
$$
\frac{N}{a},
\quad
\frac{N+a}{a},
\quad
\frac{N+2a}{a},
\quad
\frac{N+3a}{a},
\quad
\ldots
$$
coincide.

In this paper we are aiming to generalize this simple statement to the case of triples $(a,b,N)$ where $a$ and $b$ are relatively prime with $N$.
Our main tool will be Minkowski-Voronoi generalization of the continued fractions defined by
such triples.

We show that {\it the space of all triples
$($with additional constrains $b\ge 2$ and $N$ is not divisible by $b$$)$ splits into 2-dimensional families such that
the Minkowski-Voronoi continued fractions in each of these families are almost all combinatorially equivalent
to each other}, see Theorem~\ref{stabilization}, Remark~\ref{finitely_many}, and Corollary~\ref{splitting} for details.

\vspace{2mm}

We believe that this result has a strong potential to initiate a serious progress in theory of multidimensional
continued fractions (and its reflections in the singularity theory). In the first part of the paper we are aiming
to give a comprehensive introduction to theory of Minkowski-Voronoi polyhedra, including computational aspects of the theory.
The middle part is rather technical, it contains complete proof of the main theorem and the associated statements.
The last part of the paper contains experimental results which give rise to several conjectures for further investigation.

\vspace{2mm}

{
\noindent
{\bf Multidimensional continued fractions.}
The question of generalizing continued fractions to the
multidimensional case was risen by the first time by
C.~Hermite~\cite{Hermite1839} in 1839. Since than, many different
generalizations were introduced (see~\cite{KarpenkovBook} for a
general reference).
In this paper we study one of the geometrical generalizations of continued fractions
proposed independently by G.F.~Voronoi~\cite{Voronoi1896,Voronoi1952} and
H.~Minkowski~\cite{Minkowski1896,Minkowski1911}.
Minkowski-Voronoi continued fractions are used as one of the main tools for the calculations of the fundamental units of algebraic fields.
Although the approaches of G.F.~Voronoi and
H.~Minkowski are rather different (see the original
papers~\cite{Delone1940,Delone1947,Hancock1964a,Pepper1939} and the survey~\cite{Ustinov2015}), they
deal with the same geometrical object, with the set of local
minima for lattices (we provide with all necessary definitions in
Section~\ref{Definitions}). One of the most important properties
of the set of local minima is that it admits adjacency relation,
which provides the graph structure on the set of local minima.
This property is essential for the algorithmic applications of
Minkowski-Voronoi continued fractions,
see~\cite{Buchmann1985,Buchmann1985a,Buchmann1987,Buchmann1987a,Cusick1971,Furtwangler1920,Williams1980,Williams1981}.
}

Regular continued fractions are used as a tool in numerous
problems related to two-dimensional lattices. Although
Minkowski-Voronoi continued fractions are much more complicated
combinatorially than classical continued fractions, there is a
number of significant three-dimensional results related to them.
For instance, in~\cite{Cusick1971} and~\cite{Cusick1983}
T.W.~Cusick used Minkowski's algorithm for the study of ternary
linear forms. Using Minkowski-Voronoi continued fractions
G.~Ramharter proved Gruber's conjecture (i.e., Mordell's inverse
problem) on volumes of extreme (admissible centrally symmetric
with faces parallel to the coordinate axes) parallelepipeds,
see~\cite{Ramharter1996}. The beginnings of 3-dimensional analogs
of the Markov spectrum is due to H.~Davenport and
H.P.F.~Swinnerton-Dyer,
see~\cite{Davenport1943,Swinnerton-Dyer1971} (see also isolation
theorems in~\cite{Cassels1955,Cassels1997}). The connections
between Minkowski-Voronoi continued fractions and Klein polyhedra
are studied by V.A.~Bykovskii and O.N.~German,
see~\cite{Bykovskii2006,German2006}. Further there is a
three-dimensional analog of Vahlen's theorem,
see~\cite{Avdeeva2003b,Avdeeva2006b,Bykovskii1999,Ustinov2012,Ustinov2013}.
Minkowski-Voronoi diagrams are known in combinatorial commutative algebra as staircase diagrams (see~\cite{bookk} for more detains).
Statistical properties of Minkowski-Voronoi continued fractions
were studied by elementary methods in~\cite{Illarionov2012}
and by analytic methods based on Kloosterman sums, see~\cite{Ustinov2015}.

Our lodestar in this research is the following classical open problem.
\begin{problem}
Give a combinatorial classification of finite and periodic
Minkowski-Voronoi continued fractions.
\end{problem}
Although the first publications on Minkowski-Voronoi continued fractions were made more than a hundred years ago,
there is nothing known regarding this problem.
Our main result of this paper is the first step towards its solutions.

\vspace{2mm}

{
\noindent
{\bf Discrepancy of lattices.} The second problem we were keeping in mind while working on this paper is as follows.
}

\begin{problem}
Construct infinite families of finite three-dimensional
Minkowski-Voronoi continued fractions with bounded or growing at a
slow rate ``partial quotients''.
\end{problem}

{
\noindent
{\it Remark.}
Partial quotients here are the
coefficients of the transition matrices between the adjacent
relative minima. Such continued fractions correspond to Korobov
nets with small deviations.
}

\vspace{2mm}

Let us briefly recall the situation in the classical case.
There are several classes of regular finite
continued fractions with \textit{predictable}  partial quotients.
Good two-dimensional nets for numerical integration can be
constructed using fractions $a/N$ such that all partial quotients
in the continued fraction expansion are bounded by some small constant
(see for example the book~\cite{Korobov2004}). Clearly we can find
fractions with arbitrary large denominators and all partial
quotients equal to $1$ (the ratio of two consecutive Fibonacci
numbers). For denominators $N=2^m$, $3^m$ H.~Niederreiter found
numerators such that  partial quotients in the corresponding continued
fractions are bounded by $4$, see~\cite{Niederreiter1986}.

Recall that in the theory of uniform distribution \textit{discrepancy} is a measure
of the deviation of a point set from a uniform distribution (see for example~\cite{Korobov2004}).
In numerical analysis multidimensional equidistributed low--discrepancy sets are of the most importance.
It is widely known that in any dimension almost all lattices are equidistributed low--discrepancy sets.
Nevertheless there is no explicit construction known in dimension 3 or higher.
We consider our paper as a preliminary investigations in this direction in dimension~3.


Let us say a few words about discrepancy in the 2-dimensional case.
The discrepancy of a 2-dimensional lattice
$$
\Gamma(a,N)=\{m_1(1,a)+m_2(0,N)\mid m_1,m_2\in\mathbb{Z}\}
$$
depends on the value of maximal partial quotient in the continued fraction expansion for $a/N$.
Recall that the partial quotients are in bijection with the local minima of the lattice $\Gamma(a,N)$.
V.A.~Bykovskii proved a similar results for the lattices in arbitrary dimension, see
\cite{Bykovskii2002b,Bykovskii2003a}.
It is interesting to observe that these results allow to check the goodness of a given lattice in a
polynomial time, see~\cite{Bykovskii2004a,Bykovskii2011a}.

Basing on the analytical properties of local minima in
multidimensional integer lattices V.A.~Bykovskii
in~\cite{Bykovskii2012} proved essentially new bounds for
discrepancy of Korobov nets, which are conjectured to be sharp.
These results justify further investigations of  geometric and
combinatorial properties of local minima.

{
\noindent
{\bf Outline of main results.}
In this paper we are making the first step in the systematic study of Minkowski-Voronoi continued fractions,
our main contribution to the above problems is as follows.
}

\vspace{1mm}

\begin{itemize}
\item
First, we introduce a general natural construction of the Minkowski-Voronoi complex in all dimensions.

\vspace{1mm}

\item
Secondly, we prove that the two-parametric families of lattices considered in Theorem~\ref{stabilization}) has
asymptotically stable Minkowski-Voronoi continued fractions.
We have discovered this surprising phenomenon in a series of experiments that are based on the algorithm described in Section~2.

\vspace{1mm}

\item
Finally,
we give an alphabetic description of the canonical diagrams for White's lattices in dimension three (Theorem~\ref{White}),
and conjecture other alphabetic descriptions for a broader families of lattices (Conjecture~\ref{conjecture-3}).
\end{itemize}

\vspace{2mm}

%
%

{\noindent
{\bf This paper is organized as follows.}
We start in Section 1 with general notation and definitions.
We define minimal sets in the sense of Voronoi and show how to introduce the structure
of the Minkowski-Voronoi complex on the sets of all Voronoi minimal sets.
Further we describe two important representations of this complex in the three dimensional case:
Minkowski-Voronoi tessellations and canonical diagrams.
}

Further in Section 2 we describe the techniques to construct canonical diagrams
of Minkowski-Voronoi complexes for finite axial sets in general position.
We describe the algorithm that produces a canonical diagram for every finite axial set,
and therefore provides the existence of canonical diagrams.
Using this algorithm we have calculated many examples
that leads to the results of the next two sections.

Section 3 contains the main result this article.
We formulate and prove the stabilization theorem of Minkowski-Voronoi complexes
for special two-parametric families in the set of all rank-1 lattices.

Finally in Section 4 we introduce the alphabetical description of canonical diagrams for lattices.
This is a simple way to describe canonical diagram that arisen in the study of various examples.
We conjecture that in the case of the simplest families of lattices (like in the case of lattices corresponding to empty simplices
described by White's theorem, see in~\cite{WhiteCite}) it is possible to describe any canonical diagram with finitely many distinct letters.
In this section we provide several open problems and conjectures regarding the alphabetic description.

\section{Basic notions and definitions}\label{Definitions}

We start this section with general definitions of minimal subsets in the sense of Voronoi for a given set.
Further we define Minkowski-Voronoi complex associated with the set of all minimal subsets.
In particular in the three-dimensional case we describe a Minkowski-Voronoi tessellation of the plane encoding the geometric nature of the complex.
In conclusion of this section we say a few about the case of lattices.

In this paper we work with finite axial subsets of $\r^n_{\ge 0}$ in general position.
\begin{definition}
A subset $S\subset \r^n_{\ge 0}$ is said to be {\it axial} if $S$ contains points on each of the coordinate axes.
\end{definition}

\begin{definition}
We say that an axial subset $S\subset \r^n_{\ge 0}$ is {\it in general position} if the following conditions hold:

({\it i}) Each coordinate plane contains exactly $n-1$ points of $S$ none of which are at the origin;
these points are on different coordinate axes.

({\it ii}) Any plane parallel to a coordinate plane (and distinct to it) contains at most one point of $S$.
\end{definition}

\subsection{Minkowski-Voronoi complex}
Let $A$ be an arbitrary finite subset of the space $\r^n_{\ge 0}=(\r_{\ge 0})^n$.
For $i=1,\ldots, n$ we set
$$
  \max(A,i)=\max\{x_i\mid  (x_1,\ldots, x_n)\in A\}
$$
and define the parallelepiped $\Pi(A)$ as follows
$$
\Pi(A)=\{(x_1,\ldots, x_n) \mid 0\le x_i\le \max(A,i),i=1,\ldots, n\}.
$$

\begin{definition}\label{RelativeMinimium}
Let $S$ be an arbitrary finite axial subset of $\r^n_{\ge 0}$ in general position.
An element $\gamma\in S$ is called
a {\it Voronoi relative minimum} with respect to $S$
if the parallelepiped $\Pi(\{\gamma\})$ contains no points of $S\setminus\{\gamma\}$.
The set of all Voronoi relative minima of $S$ we denote by $\vrm(S)$.
\end{definition}

\begin{definition}\label{def-minimal}
Let $S$ be an arbitrary finite axial subset of $\r^n_{\ge 0}$ in general position.
A finite subset $F\subset \vrm(S)$ is called {\it minimal} if the following two conditions hold:

{\it $($i$)$} parallelepiped $\Pi(F)$ contains no points of $\vrm(S)\setminus F$;

{\it $($ii$)$} all the points of $S$ are contained at the boundary of $\vrm(S)$.
\\
We denote the set of all minimal $k$-element subsets of $\vrm(S)\subset S$ by $\mathfrak{M}_k(S)$.
\end{definition}

It is clear that any minimal subset of a minimal subset is also minimal.

\begin{definition}
Consider a finite subset $S$ in general position.
A {\it Minkowski-Voronoi complex} $MV(S)$ is an $(n-1)$-dimensional complex such that

{\it $($i$)$} the $k$-dimensional faces of $MV(S)$ are enumerated by its minimal $(n{-}k)$-element subsets
(i.e., by the elements of $\mathfrak{M}_{n-k}(S)$);

{\it $($ii$)$} a face with a minimal subset $F_1$ is adjacent to a face with a minimal subset $F_2\ne F_1$ if
and only if $F_1\subset F_2$.
\end{definition}

\begin{remark}
In the three-dimensional case it is also common to consider the Voronoi and Minkowski graphs
that are subcomplexes of the Minkowski-Voronoi complex. They are defined as follows.

\vspace{1mm}

{\noindent
The {\it Voronoi graph}\index{graph!Voronoi} is the graph whose vertices and edges are
respectively vertices and edges of the Minkowski-Voronoi complex.
}

{\noindent
The {\it Minkowski graph}\index{graph!Minkowski} is the graph whose vertices and edges are
respectively faces and edges of the Minkowski-Voronoi complex
(two vertices in the Minkowski graph are connected by an edge if and only if
the corresponding faces in the Minkowski-Voronoi complex have a common edge.
}

%
\end{remark}

\begin{example}\label{ExVoronoi}
Let us consider an example of a 6-element set $S_0\subset\r^3_{\ge 0}$ defined as follows
$$
S_0=\big\{ \gamma_1,\gamma_2,\gamma_3,\gamma_4,\gamma_5,\gamma_6\big\},
$$
where
$$
\begin{array}{ccc}
\gamma_1=(3,0,0),& \gamma_2=(0,3,0),&\gamma_3=(0,0,3),\\
\gamma_4=(2,1,2),& \gamma_5=(1,2,1), &\gamma_6=(2,5,6).\\
\end{array}
$$
There are only five Voronoi relative minima for the set $S_0$, namely the vectors $\gamma_1,\ldots, \gamma_5$.
The Minkowski-Voronoi complex contains 5 vertices, 6 edges, and 5 faces.
The vertices of the complex are
$$
\begin{array}{c}
v_1=\{\gamma_1,\gamma_3,\gamma_4\},\quad
v_2=\{\gamma_3,\gamma_4,\gamma_5\},\quad
v_3=\{\gamma_1,\gamma_4,\gamma_5\},\\
v_4=\{\gamma_2,\gamma_3,\gamma_5\},\quad
v_5=\{\gamma_1,\gamma_2,\gamma_5\}.\\
\end{array}
$$
Notice that the set $\{\gamma_2,\gamma_4,\gamma_5\}$ is not minimal by Definition~\ref{def-minimal}(ii), since $\gamma_5$ is in the interior of $\vrm(\{\gamma_2,\gamma_4,\gamma_5\})$.
Its edges are
$$
\begin{array}{ccc}
e_1=\{\gamma_1,\gamma_3\},&
e_2=\{\gamma_3,\gamma_2\},&
e_3=\{\gamma_1,\gamma_2\},\\
e_4=\{\gamma_3,\gamma_4\},&
e_5=\{\gamma_1,\gamma_4\},&
e_6=\{\gamma_4,\gamma_5\},\\
e_7=\{\gamma_3,\gamma_5\},&
e_8=\{\gamma_1,\gamma_5\},&
e_9=\{\gamma_2,\gamma_5\}.\\
\end{array}
$$
Its faces are
$$
f_1=\{\gamma_1\},\quad
f_2=\{\gamma_2\},\quad
f_3=\{\gamma_3\},\quad
f_4=\{\gamma_4\},\quad
f_5=\{\gamma_5\}.
$$
Finally, we represent the complex $MV(S)$ as a tessellation of an open two-dimensional disk.
We show vertices (on the left), edges (in the middle), and faces (on the right) separately:
$$
\includegraphics{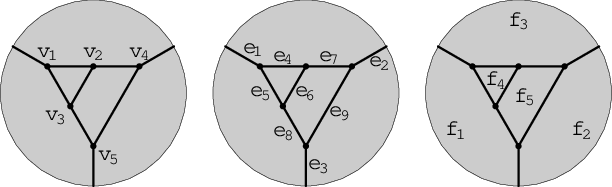}
$$
\end{example}

\subsection{Minkowski-Voronoi tessellations of the plane}

In this subsection we discuss a natural geometric construction standing behind the Minkowski-Voronoi complex
in the three-di\-men\-sional case.

\begin{definition}
Let $S$ be an arbitrary finite axial subset of $\r^3_{\ge 0}$ in general position.
The {\it Minkowski polyhedron} for $S$ is the boundary of the set
$$
S\oplus \r^3_{\ge 0}=\{s+r\mid s\in S, r\in\r^3_{\ge 0}\}.
$$
\end{definition}

In other words, the Minkowski polyhedron is the boundary of the union
of copies of the positive octant shifted by vertices of the set $S$.

\begin{proposition}
The union of the compact faces of the Minkowski polyhedron is contained in
$$
\partial \Big(\bigcup\limits_{A\in \mathfrak{M}_{3}(S)} \Pi(A)\Big).
$$
\qed
\end{proposition}

\begin{definition}\label{Minkowski-Voronoi tessellation}
The {\it Minkowski-Voronoi tessellation}
for a finite axial set $S\subset\r^3_{\ge 0}$ in general position is a tessellation of the plane $x+y+z=0$
obtained by the following three steps.

{\noindent
{\it Step 1.} Consider the Minkowski polyhedron for the set $S$ and project it orthogonally to the plane $x+y+z=0$.
This projection induces a tessellation of the plane by edges of the Minkowski polyhedron.
}

{\noindent
{\it Step 2.} From the tessellation of Step~1 remove all vertices corresponding to the local minima of the
function $x+y+z$ on the Minkowski polyhedron (these are exactly the images of the relative minima of $S$ under the projection).
Remove also all edges which are adjacent to the removed vertices.
}

{\noindent
{\it Step 3.} After Step~2 some of the vertices are of valence~1.
For each vertex $v$ of valence $1$ and the only remaining edge $wv$ with endpoint at $v$ we replace
the edge $wv$ by the ray $wv$ with vertex at $w$ and passing through $v$.
}
\end{definition}

\begin{proposition}
The Minkowski-Voronoi tessellation for a finite axial set $S$ in general position has the combinatorial structure of
the Minkowski-Voronoi complex $MV(S)$.
\qed
\end{proposition}


\begin{example}\label{ExVoronoi2}
Consider the set $S_0$ as in Example~\ref{ExVoronoi}.
In Figure~\ref{mv-7} we show the Minkowski polyhedron (on the left) and the corresponding
Minkowski-Voronoi tessellation (on the right).
The local minima of the function $x+y+z$ on the Minkowski polyhedron for $S_0$ are in bijection with the the relative minima
$f_1,\ldots, f_5$. Without loss of generality, we denote the local minima by $\hat f_1, \ldots, \hat f_5$:
$$
\begin{array}{c}
\hat f_1=(3,0,0), \quad \hat f_2=(0,3,0), \quad \hat f_3=(0,0,3),\\
\hat f_4=(2,1,2), \quad \hat f_5=(1,2,1).
\end{array}
$$
They identify the faces of the complex $MV(S_0)$.

The local maxima of the function $x+y+z$ on the Minkowski polyhedron for $S_0$ are
$\hat v_1,\ldots, \hat v_5$, corresponding to the vertices $v_1,\ldots, v_5$ of the complex $MV(S_0)$.
The vertices $\hat v_1,\ldots, \hat v_5$ are as follows:
$$
\begin{array}{c}
\hat v_1=(3,1,3),\quad
\hat v_2=(2,2,3),\quad
\hat v_3=(3,2,2),\\
\hat v_4=(1,3,3),\quad
\hat v_5=(3,3,1).
\end{array}
$$

\begin{figure}[t]
$$\includegraphics{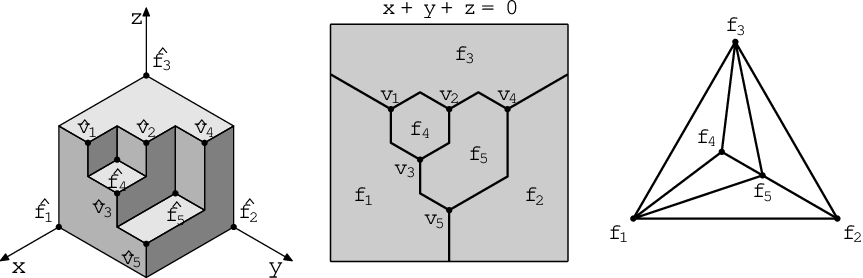}$$
\caption{The Minkowski polyhedron (on the left), the corresponding
Minkowski-Voronoi tessellation (in the middle), and the Minkowski graph (on the right).}\label{mv-7}
\end{figure}
\end{example}

\subsection{Canonical diagrams of three-dimensional Minkowski-Voronoi complexes for finite axial sets in general position}
Let us first describe a canonical labeling of edges and vertices of the Minkowski-Voronoi complex.

\subsubsection{Labels for edges and vertices}

Consider a minimal 3-element subset $\{\gamma_1,\gamma_2,\gamma_3\}$, it is a vertex of the Minkowski-Voronoi complex.
There are exactly three edges that are adjacent to this vertex.
They are enumerated by minimal 2-element subsets $\{\gamma_1,\gamma_2\}$,  $\{\gamma_1,\gamma_3\}$, and  $\{\gamma_2,\gamma_3\}$.
Hence there are exactly three vertices (this may include a vertex at infinity)
that are connected with $\{\gamma_1,\gamma_2,\gamma_3\}$ by an edge.
In each of these vertices the corresponding minimal 3-element subset has exactly two elements in $\{\gamma_1,\gamma_2,\gamma_3\}$.
Without loss of generality we consider one of them which is $\{\gamma_1,\gamma_2,\gamma_3'\}$.
Notice that
$$
\Pi(\{\gamma_1,\gamma_2,\gamma_3\})\cap\Pi(\{\gamma_1,\gamma_2,\gamma_3'\})=\Pi(\{\gamma_1,\gamma_2\}).
$$
Hence the parallelepipeds $\Pi(\{\gamma_1,\gamma_2,\gamma'_3\})$ is obtained from $\Pi(\{\gamma_1,\gamma_2,\gamma_3\})$
by increasing one of its sizes and by decreasing some other.
This gives a natural coloring of each edge pointing out of the vertex $(\{\gamma_1,\gamma_2,\gamma_3\})$ into six colors
(each color indicate which coordinate we increase, which we decrease, and which stays unchanged).

\begin{definition}\label{reflabel}
To each color we associate one of the directions $\frac{k\pi}{3}$ (where $k=0,1,\ldots,5$)
according to the scheme shown on Figure~\ref{faces-3} from the left.
\begin{figure}[t]
$$\includegraphics{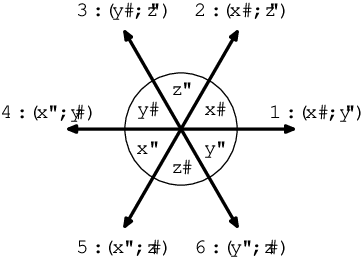}\qquad\qquad\includegraphics{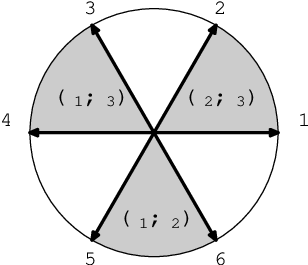}$$
\caption{Six different directions corresponding to different colors.}\label{faces-3}
\end{figure}
We call such direction of an edge the {\it labeling} of this edge.

The {\it labeling} of a vertex is a collection of all labelings for all finite edges adjacent to this vertex.
\end{definition}

\begin{example}\label{ExVoronoi3}
Consider the set $S_0$ as in Example~\ref{ExVoronoi} (and in Example~\ref{ExVoronoi2}).
By $[x,y,z]$ we denote the parallelepiped whose vertices are $(\pm x, \pm y, \pm z)$.
Then the parallelepipeds corresponding to the vertices $v_1,\ldots, v_5$ are respectively
$$
\begin{array}{c}
\Pi_1=[3,1,3],\quad
\Pi_2=[2,2,3],\quad
\Pi_3=[3,2,2],\\
\Pi_4=[1,3,3],\quad
\Pi_5=[3,3,1].
\end{array}
$$
Consider the oriented edge $v_1v_2$, recall that in notation of Example~\ref{ExVoronoi} it is $e_4=\{\gamma_3,\gamma_4\}$.
The parallelepiped $\Pi_2=[2,2,3]$ is obtained from $\Pi_1=[3,1,3]$ be decreasing $x$ coordinate and increasing $y$ coordinate.
Hence the label for this edge is $(x\downarrow, y\uparrow)$.

In a similar way we find labels for the remaining five compact edges (here we follow the notation of Example~\ref{ExVoronoi}). We have:
\begin{itemize}
\item The edge $e_5=\{\gamma_1,\gamma_4\}$ connects $v_1$ with $v_3$. The label is $(y\uparrow,z\downarrow)$.

\item The edge $e_6=\{\gamma_4,\gamma_5\}$ connects $v_2$ with $v_3$. The label is $(x\uparrow,z\downarrow)$.

\item The edge $e_7=\{\gamma_3,\gamma_5\}$ connects $v_2$ with $v_4$. The label is $(x\downarrow,y\uparrow)$.

\item The edge $e_8=\{\gamma_1,\gamma_5\}$ connects $v_3$ with $v_5$. The label is $(y\uparrow,z\downarrow)$.

\item The edge $e_9=\{\gamma_2,\gamma_5\}$ connects $v_4$ with $v_5$. The label is $(x\uparrow,z\downarrow)$.
\end{itemize}

\vspace{2mm}
Finally the edges $e_1$, $e_2$, $e_3$ are rays, one can say that {\it they have one vertex at infinity}; we do not define labels for them.
\end{example}

\subsubsection{Geometric structure of vertex-stars}

Without loss of generality we suppose that $\gamma_1$ has the greatest $x$-coordinate among the $x$-coordinates of $\gamma_1$, $\gamma_2$, and $\gamma_3$;
let $\gamma_2$ have the greatest $y$-coordinate, and let $\gamma_3$ have the greatest $z$-coordinate.
There is only one edge adjacent to $\{\gamma_1,\gamma_2,\gamma_3\}$ along which the first coordinate is decreasing,
it is $\{\gamma_2,\gamma_3\}$. Therefore, each vertex that is not adjacent to infinite edges has exactly one edge in
direction 1 or 2  (see Figure~\ref{faces-3} from the right). Similarly, each vertex has exactly only edge in
directions 3 or 4, and one edge in directions 5 or 6.
Hence the following statement is true.

\begin{proposition}\label{types}
Each vertex of the Minkowski-Voronoi complex that is not adjacent to infinite edges
has one of the following eight labels
$$
\includegraphics{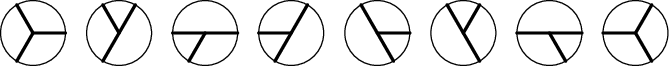}.
$$
\qed
\end{proposition}

\subsubsection{Definition of canonical diagrams}

Labeling of vertices and edges gives rise to the following geometric definition.

\begin{definition}
Consider a finite axial set $S\subset \r^3_{\ge 0}$ in general position.
We say that an embedding of the 1-skeleton of the Minkowski-Voronoi complex $MV(S)$
to the plane with linear edges (some of them are infinite rays corresponding to infinite edges)
is a {\it canonical} diagram of $S$ if the following conditions hold:

--- All finite edges are straight segments in the directions of their labels  (i.e, $\frac{k\pi}{3}$ for some $k\in \z$).

--- Each vertex that does not contain infinite edges is a vertex of one of 8 types described in Proposition~\ref{types}.
\end{definition}

\begin{example}
Consider the Minkowski-Voronoi complex for the set:
$$
S=\{(6,0,0),(0,6,0),(0,0,6), (1,2,3), (2,4,2),(3,3,1), (4,1,4)\}.
$$
It is shown on Figure~\ref{faces-2} from the left (we skip the construction here).
Each vertex of the complex in the picture is represented by an appropriate label.
Notice that we have exactly three vertices adjacent to infinite edges.
On Figure~\ref{faces-2} from the right we show the canonical diagram of this Minkowski-Voronoi complex.

\begin{figure}[t]
$$\includegraphics{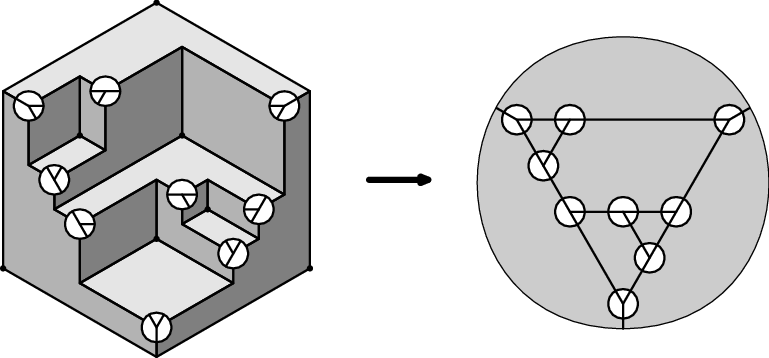}$$
\caption{A Minkowski-Voronoi complex and its canonical diagram.}\label{faces-2}
\end{figure}
\end{example}

In Section~\ref{Canonical diagrams} (Theorem~\ref{MinTheorem}) below we prove the existence
of canonical diagrams for the case of finite axial subsets of $\r^3_{\ge 0}$ in general position
and give an algorithm to construct them.

\vspace{2mm}

\subsubsection{Geometry of faces in canonical diagrams}

Consider the Minkowski polyhedron for a finite axial subset $S\subset \r^3_{\ge 0}$ in general position.
Let $\gamma_0$ be any of the local minima of $S$  not contained in the coordinate planes.
Then there are exactly three faces of the Minkowski polyhedron that meet at this minimum:
each of such faces is parallel to one of the coordinate planes (see Figure~\ref{faces.1} from the left).
The boundary of their union consists of three ''staircases'' in the corresponding planes.
Each staircase may have an arbitrary nonnegative number of stairs.
On Figure~\ref{faces.1} from the left such staircases have 0, 2, and 4 stairs respectively.
By Minkowski convex body theorem, each staircase has only finitely many stairs.

\begin{figure}[t]
$$\includegraphics{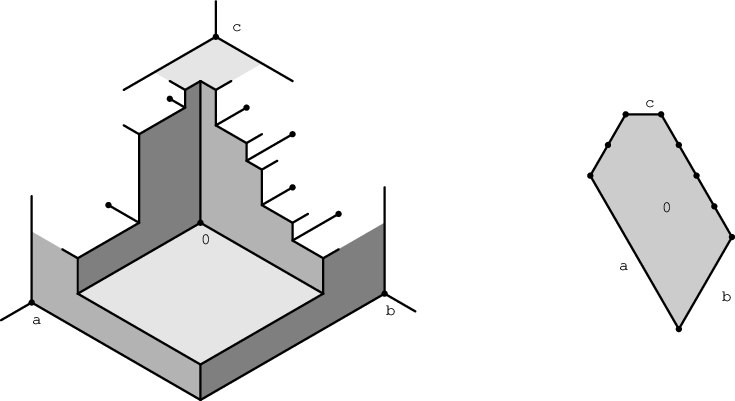}$$
\caption{The union of three faces of Minkowski polyhedron representing $\{\gamma_0\}$ (left);
the canonical diagram for $\{\gamma_0\}$ (right)}\label{faces.1}
\end{figure}

So there is a natural cyclic order for all the relative minima that are neighbors of $\gamma_0$.
There are three such neighbor minima for $\gamma_0$; these are the minima that we capture if we
start to increase one of the dimensions of $\Pi(\gamma_0)$, namely $x$-, or $y$-, or $z$-dimension.
We denote such minima by $\gamma_a$, $\gamma_b$, and $\gamma_c$ respectively (here we do not consider the boundary case
when one of such minima does not exist).
All the other neighboring minima are obtained by increasing simultaneously some two dimensions of the parallelepiped
$\Pi(\gamma_0)$.

\begin{theorem}\label{face-shape}
Each finite face in a canonical diagram has a combinatorial type
$$\includegraphics{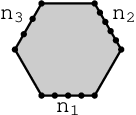}$$
for some nonnegative integers $n_1,n_2,n_3$ $($where $n_1$, $n_2$, and $n_3$ are the number of segments on the corresponding edges$)$.
\qed
\end{theorem}

Notice that if all $n_0$, $n_1$, and $n_2$ are zeroes, we have the simplest possible type of face -- the $\nabla$-shaped triangle.


\section{Construction of canonical diagrams}\label{Canonical diagrams}

In this section we describe the algorithm to construct canonical diagrams of Minkowski-Voronoi complexes for finite axial sets in general position.
We do not have an intention to optimize the location of points in the diagram, so some faces in them
may become quite narrow. In order to improve the visualization of the diagram itself one should apply
``Schnyder wood'' techniques, see~\cite{Felsner2008}.
The algorithm of this section always returns a canonical diagram, which provides the existence of
canonical diagrams. We formulate this statement as follows.

\begin{theorem}\label{MinTheorem}
Let $S$ be a finite axial subset of $\r^3_{\ge 0}$ in general position.
Then $MV(S)$ admits a canonical diagram.
\end{theorem}

\begin{proof}
One of the canonical diagram is explicitly defined by the algorithm described below.
\end{proof}

Let us first introduce some necessary notions and definitions.

We assume that the set $S$ is an axial set in general position. In particular this means
that $S$ contains three points $(N_1,0,0)$, $(0,N_2,0)$, and $(0,0,N_3)$ for some $N_1,N_2, N_3>0$.
Notice that for every two of these three points there is exactly one minimal triple of $S$ containing them. Such triples correspond to
three vertices of $MV(S)$ which we denote by
$v_L$, $v_R$, and $v_B$ (right, left, and bottom vertices) where:
$$
\begin{array}{c}
\{(0,N_2,0), (0,0,N_3)\} \subset v_R,\\
\{(N_1,0,0), (0,0,N_3)\} \subset v_L,\\
\{(N_1,0,0), (0,N_2,0)\} \subset v_B.
\end{array}
$$

Set
$$
e_1=\Big(-\frac{1}{2},\frac{\sqrt{3}}{2}\Big);\qquad
e_2=\Big(\frac{1}{2},\frac{\sqrt{3}}{2}\Big);\qquad
e_3=(1,0).
$$
Recall that all the edges in canonical diagrams have directions $\pm e_1$, $\pm e_2$, or $\pm e_3$.

By a {\it directed path} in the 1-skeleton of the Minkowski-Voronoi complex
we consider the sequence of vertices of $MV(S)$ such that each two consequent vertices are connected by an edge of $MV(S)$.

\begin{definition}
We say that a directed path $p_0p_1\ldots p_n$ in the 1-skeleton of the Minkowski-Voronoi complex $MV(S)$
is {\it ascending} if the following conditions are fulfilled

{$($\it i$)$} $p_0=v_B$;

{$($\it ii$)$} $p_n=v_L$;

{$($\it iii$)$} for every $i=0,\ldots, n-1$ the edge $p_ip_{i+1}$ is of one of the following three directions: $e_1$, $e_2$, or $-e_3$.
\end{definition}

\begin{definition}
Let us consider the union of all compact faces of $MV(S)$ and an ascending path $P$.
The path $P$ divides the union of all compact faces into several connected components.
We say that a compact face, an edge not contained in $P$, or a vertex not contained in $P$ is {\it to the right} (or {\it to the left})
of the path $P$ if it is (or, respectively, it is not) in the same connected component with the point $v_R$.
In the exceptional case when $v_R$ is a vertex of $P$ we say that all compact faces, edges (not in $P$), and vertices (not in $P$)
are to the left of $P$.
\end{definition}

\begin{figure}[t]
$$\includegraphics{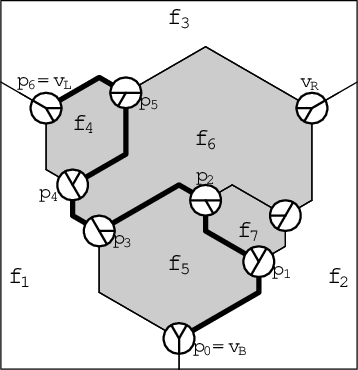}$$
\caption{A Minkowski-Voronoi tesselation with labelled vertices.}\label{faces-6}
\end{figure}

\begin{example}
On Figure~\ref{faces-6} we consider an example of Minkowski
--Voronoi tessellation.
It contains three noncompact faces $f_1$, $f_2$, $f_3$ and four compact faces $f_4$, $f_5$, $f_6$, $f_7$ (all compact faces are in gray color).
The path $P=p_0p_1p_2p_3p_4p_5p_6$ is ascending.
It divides the union of all compact parts into three connected components.
The faces $f_6$ and $f_7$ are to the right of $P$, while the faces $f_4$ and $f_5$ are to the left of $P$.
\end{example}

Let us bring together the main properties of ascending paths.

\begin{proposition}\label{ascending}
Let  $P=p_0p_1\ldots p_n$ be an ascending path.

{$($\it i$)$} Let $\Pi(p_0), \ldots, \Pi(p_n)$ be the sequence of parallelepipeds defined by the triples of points corresponding to $p_i$ $($$i=0,\ldots, n$$)$.
Then, first, the sequence of $y$-sizes $($i.e., maximal $y$-coordinates$)$ of such parallelepipeds is non-increasing. Second, the sequence of
$z$-sizes of the parallelepipeds is non-decreasing.

{$($\it ii$)$} There exists only one ascending path that contains the vertex $v_R$.

{$($\it iii$)$} If the path does not contain $v_R$, then there exists a vertex $w$ outside the path and an integer  $i$ such that
the edge $p_iw$ is an edge of Minkowski-Voronoi complex in direction $e_2$.

{$($\it iv$)$}
All edges starting at some $p_i$ to the right of $P$ are either in direction $e_2$ or in direction $e_3$.

\end{proposition}

\begin{proof}
{$($\it i$)$}. The first item follows directly from the definition of labeling (Definition~\ref{reflabel}).

\vspace{2mm}

{\noindent
{$($\it ii$)$}. First of all, let us construct the following directed path in the 1-skeleton of $MV(S)$.
This path consists of two parts.
In the first part of the path we collect all vertices  (without repetitions) that include the point $(0,N_2,0)$; all such vertices are
consequently joined by the edge in direction $e_2$ starting with $v_B$ ending with $v_R$.
The second part of the path consists of vertices (passed once) that include the point $(0,0,N_3)$; all such vertices are
consequently joined by the edge in direction $-e_3$ starting with $v_R$ ending with $v_L$.
By construction this path is ascending.
}

Suppose now we have some other ascending path through $v_R$.
Let us first consider the part of the path between $v_B$ and $v_R$.
The $y$-sizes of $\Pi(v_B)$ and $\Pi(v_R)$ are equal to $N_2$, and hence
by Proposition~\ref{ascending}(i) all the $y$-sizes of all the parallelepipeds within the path between $v_B$ and $v_R$
are equal to $N_2$. For the edges with directions $e_1$ and $-e_3$ the $y$-size of the parallelepiped strictly increases,
and hence there are none of them in the part of the path between $v_B$ and $v_R$. Hence all the edges are in direction $e_2$,
and, therefore, this part of the path coincides with the first part of the ascending path constructed before.

The second part of the path is between $v_R$ and $v_L$. Here the $z$-sizes of all the corresponding parallelepipeds are equal to $N_3$
(since by Proposition~\ref{ascending}(i) the sequence of $z$-sizes is non-decreasing and since all $z$-sizes are bounded by $N_3$ from above).
Therefore, this part of the path does not contain the directions $e_1$ and $e_2$ that increase $z$-sizes.
Hence all the directions are $-e_3$, and this part of the path coincides with the second part of the considered before path.

Therefore, there exists and unique an ascending path through $v_R$.

\vspace{2mm}

{\noindent
{$($\it iii$)$}. Consider an ascending path that does not contain $v_R$. Let us find the last vertex that contains
$(0,N_2,0)$. Since it is not $v_R$, there is an edge in direction $e_2$ at this point.
This edge does not change the  $y$-size of the parallelepiped and hence it is not in the ascending path.
}

\vspace{2mm}

{\noindent
{$($\it iv$)$}. The properties of being from the right or being from the left are detected from the Minkowski-Voronoi tessellation for the corresponding
set. Hence it is enough to prove the statement locally, considering only eight possible labels of the vertices shown in the figure of Proposition~\ref{types}.
It is clear that all possible ascending paths trough all such vertices either do not contain edges directed to the right, or the directions
are either $e_2$ or $e_3$.
}
\end{proof}

Now we outline the main steps of the algorithm.

\begin{center}

{\bf Algorithm to construct a canonical diagram of the Minkowski-Voronoi complex}

\end{center}

\vspace{2mm}

{
\noindent
{\bf Input data.} We are given by a finite strict subset $S$ of $\r^3_{\ge 0}$ in general position
containing three points $(N_1,0,0)$, $(0,N_2,0)$, and $(0,0,N_3)$.
}

\vspace{2mm}

{
\noindent
{\bf Goal of the algorithm.}
To draw the canonical diagram for the set $S$.
}

\vspace{2mm}

{
\noindent
{\bf Preliminary Step i.} First of all we construct the Minkowski-Voronoi complex $MV(S)$.
This is done by comparing coordinates of the points of the set $S$.
In the output of this step we have:
\begin{itemize}
\item the list of all Voronoi relative minima $v_1,\ldots, v_{n(v)}$;
\item the list of all edges $e_1,\ldots, e_{n(e)}$ of $MV(S)$;
\item the list of all faces $f_1,\ldots, f_{n(f)}$ of $MV(S)$;
\item the adjacency table for the complex $MV(S)$;
\item all vertices and edges of the Minkowski polyhedron.
\end{itemize}
}

{
\noindent
{\bf Preliminary Step ii.} After the previous step is completed we have all the data to construct the Minkowski-Voronoi tessellation.
In order to do this we apply the algorithm of Definition~\ref{Minkowski-Voronoi tessellation} to the Minkowski polyhedron.
In the output of this step we have:
\begin{itemize}
\item the directions of all the edges;
\item for every ascending path $P$ and for every face $f_i$ we know whether $f_i$ is to the left of the path $P$ or not.
\end{itemize}
}

{
\noindent
{\bf Step 1.} We start with the bottom vertex $v_P=p_0$.
We assign to $p_0$ the coordinates: $(0,0)$.
Further, we construct the ascending path of vertices $p_0,p_1, \ldots, p_n$
of all vertices whose triples include the point $(N_1,0,0)$ of the set $S$.
The corresponding coordinates are assigned as follows
$$
p_i=\frac{i}{n}e_1.
$$
}

Denote by $F$ the set of compact faces for which we have already constructed the coordinates.
At this step $F$ is empty. In addition we choose the path
$P=p_0\ldots p_n$. It is clear that the path $P$ is ascending.

\vspace{2mm}

{\noindent
{\bf ШRecursion Step.} At each step we start with the set $F$ of all constructed faces and an ascending path
$P=p_0\ldots p_n$, which separates the union of the constructed faces and all the other faces. All faces of $F$
are the faces to the left of $P$ (like faces $f_5$ and $f_6$ on Figure~\ref{faces-6}).
}

If $P$ contains the vertex $v_R$, then the algorithm terminates, by Proposition~\ref{ascending}(ii) we have
constructed a canonical diagram for $MV(S)$.

Suppose now that $P$ does not contain $v_R$. Then by Proposition~\ref{ascending}(iii) there exists an edge in direction $e_2$
with an endpoint at some $p_i$.
Let $p_k$ be such vertex with the greatest possible $k$, and denote by $w_1$ the other endpoint of the edge starting at $p_k$ with direction $e_2$.
Denote by $f$ the face adjacent to $p_kw_1$ from the left side. We have not yet constructed the face $f$ since it is to the right of $P$.
Suppose that $p_k,\ldots, p_{k+j}$ are vertices of $f$ and $p_{k+j+1}$ is not a vertex of $f$.

By Proposition~\ref{types} there exists an edge leaving the vertex $p_k$ either in direction $e_1$ or in direction $-e_3$.
In both cases this edge is adjacent to the same face as the edge $p_kw_1$, and in both cases this is the only choice for the edge $p_kp_{k+1}$.
Therefore, $j\ge 1$, or in other words the edge $p_kp_{k+1}$ is the edge of $f$.

We enumerate the vertices of $f$ clockwise as follows
$$
p_k,\ldots, p_{k+j},w_m,\ldots,w_1.
$$
The edge $p_{k+j}w_m$ is to the right of $P$, hence by Proposition~\ref{ascending}(iv) it is either in direction $e_2$ or in direction $e_3$.
By the assumption the vertices $p_{k+1}, \ldots, p_n$ do not have adjacent edges in direction $e_2$. Hence the direction of $p_{k+j}w_m$ is $e_3$.
Now Theorem~\ref{face-shape} (on general face structure) implies that all vertices $w_1, \ldots, w_m$ are in a line with direction $e_1$.
Since $f$ is to the right of $P$, all vertices $w_1, \ldots, w_m$ are to the right of $P$ as well.

Let us assign the coordinates to the vertices $w_1,\ldots, w_m$.
Let $u$ be the intersection point of the line containing $p_k$ and parallel to $e_2$ and the line containing $p_{k+j}$ and parallel to $e_3$.
In case when $m=1$ we assign $w_1=u$.
Suppose $m>1$. Consider two subsegments of $uv_k$ and $uv_l$ with endpoints $u$ respectively and with length
$$
\frac{\min(|uv_k|,|uv_l|)}{2}.
$$
Denote the other two endpoints of such segments by $u_k$ and $u_l$ respectively.
Subdivide $u_ku_l$ into $m-1$ segments of the same length, and assign to $w_1, \ldots, w_m$  the corresponding consequent endpoints
of the subdivision.
This gives the expressions for the assigned coordinates of all new points $w_1, \ldots, w_m$.

The recursion step is completed. We go to the next recursion step with the new data $F'=F\cup \{f\}$ and
$$
P'=(p_0,\ldots, p_k, w_1,\ldots, w_m, p_{k+j},\ldots, p_n).
$$
First, the difference between $P$ and $P'$ is that the face $f$ is to the right of $P$ and to the left of $P'$. Hence $F'$ is the set of all faces
that are to the left of $P'$ (it is the set of all constructed faces).
Second, since all the added edges in the path follow directions $e_2$, $e_1$, and $-e_3$, the path $P'$ is ascending.
Hence we are in correct settings for the next step.

\begin{remark}
The number of recursion steps coincides with the number of relative minima, which is the number of faces in our diagram.
Hence the algorithm stops in a finite time.
\end{remark}

\begin{example}
Consider the example of the Minkowski-Voronoi complex on Figure~\ref{faces-2} from the left
and let us show how to obtain its canonical diagram using the above algorithm.

\begin{figure}[t]
$$
\begin{array}{l}
\includegraphics{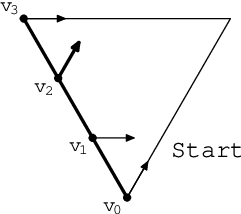}\quad \includegraphics{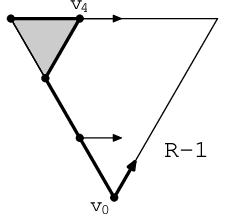}\quad \includegraphics{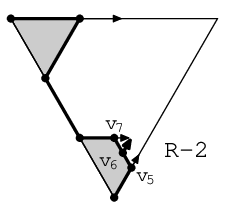}\\
\includegraphics{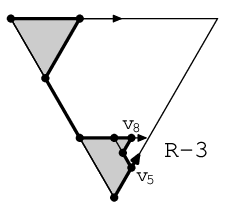}\quad \includegraphics{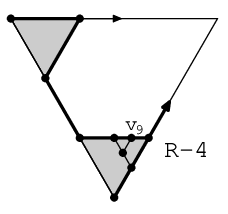}\quad \includegraphics{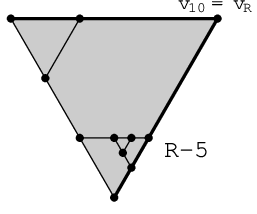}
\end{array}
$$
\caption{Six iterations of the algorithm for a particular example.}\label{alg-ex}
\end{figure}

Assume that Preliminary Steps~i and~ii are completed. So we have the Minkowski-Voronoi tessellation with labeled vertices as on Figure~\ref{faces-6}. Recall that
$$
e_1=\Big(-\frac{1}{2},\frac{\sqrt{3}}{2}\Big);\qquad
e_2=\Big(\frac{1}{2},\frac{\sqrt{3}}{2}\Big);\qquad
e_3=(1,0).
$$

\vspace{2mm}

{\noindent
{\bf Step 1 (Start):}
First of all we find the coordinates of the vertices for the ascending path for which all the faces are to the right.
We have 4 vertices here. We associate the following coordinates for them
$$
v_0=(0,0); \qquad v_1=\frac{1}{3}e_1; \qquad v_2=\frac{2}{3}e_1; \qquad v_1=e_1.
$$
So we get the points shown on Figure~\ref{alg-ex}(Start).
}

\vspace{2mm}

{\noindent
{\bf Recursive Step 1 (R-1):}  There are two vertices that have an edge to the right in direction $e_2$, see Figure~\ref{alg-ex}(Start).
They are $v_0$ and $v_2$.
According to the algorithm we take the one with the greater index, i.e., $v_2$.
We add the new vertex
$$
v_4=e_1+\frac{1}{3}e_3.
$$
We arrive at Figure~\ref{alg-ex}(R-1).
}

\vspace{2mm}

{\noindent
{\bf Recursive Step 2 (R-2):}
There is only one vector that has an edge to the right in direction $e_2$, see Figure~\ref{alg-ex}(R-1),
it is $v_0$.
So we add the following three vertices
$$
v_5=\frac{1}{6}e_2; \qquad v_6=\frac{1}{6}e_2+\frac{1}{12}e_1; \qquad v_7=\frac{1}{6}e_2+\frac{1}{6}e_1; \qquad
$$
and construct the face $v_0v_5v_6v_7v_1$ of the diagram, see Figure~\ref{alg-ex}(R-2).
}

\vspace{2mm}

{\noindent
{\bf Recursive Step 3 (R-3):}
Here we again have two vertices with edges to the right in direction $e_2$, see Figure~\ref{alg-ex}(R-2),
they are $v_6$ and $v_7$. We choose the one with the greater index, which is $v_7$.
So we add one triangular face with a new vertex
$$
v_8=\frac{1}{3}e_2-\frac{1}{12}e_3,
$$
see Figure~\ref{alg-ex}(R-3).
}

\vspace{2mm}

{\noindent
{\bf Recursive Step 4 (R-4):}
Now we add the vertex
$$
v_9=\frac{1}{3} e_2.
$$
see Figure~\ref{alg-ex}(R-4).
}

\vspace{2mm}

{\noindent
{\bf Recursive Step 5 (R-5):}
Finally we add the last vertex
$$
v_{10}=v_R=e_2.
$$
The algorithm terminates here.
We have constructed the whole canonical diagram (see Figure~\ref{alg-ex}(R-5)).
}
\end{example}


\section{Theorem on stabilization of Minkowski-Voronoi complex}\label{stab}

In this section we formulate and prove the main result of the paper.
In Subsection~\ref{MV-lat} we briefly discuss how to adapt all the definition to the case of lattices.
Further in Subsection~\ref{geo_code} we describe a geometric code for pairs and triples of integers,
we further use this code in the formulation of our main result.
We state the stabilization theorem in Subsection~\ref{MainResultFormulation} and further solve it
in Subsections~\ref{Classification of relative minima}, \ref{Asymptotic comparison of coordinate functions},
and~\ref{Conclusion of the proof}.

\subsection{Minkowski-Voronoi complexes for lattices}\label{MV-lat}

Consider a lattice $\Gamma\in \r^n$ defined as follows:
$$
\Gamma=\Big\{\sum_{i=1}^n m_ig_i\hbox{ }\Big|\hbox{ } m_i\in\mathbb{Z},i=1,\ldots, n\Big\},
$$
where $g_1, \ldots, g_n$ are linearly independent vectors in $\r^n$.
Define
$$
|\Gamma|=\{(|x_1|,\ldots,|x_n|)\mid (x_1,\ldots, x_n)\in\Gamma\}\setminus\{(0,\ldots,0)\}.
$$

\begin{definition}
Consider an arbitrary full rank lattice $\Gamma$ in $\r^n$.
Let the set of all Voronoi relative minima of $|\Gamma|$ (i.e., the set $\vrm(|\Gamma|)$
is a finite axial set.
Then the complex $MV(|\Gamma|)$ be well defined. It is called
the {\it Minkowski-Voronoi complex} for $\Gamma$, we denote it by $MV(\Gamma)$.
\end{definition}

Recall a general definition of 1-rank lattices.

\begin{definition}\label{1-rank}
Let $a$, $b$, and $N$ be arbitrary positive integers.
The lattice
$$
  \{(m_1(1,a,b)+m_2(0,N,0)+m_3(0,0,N)\mid m_1,m_2,m_3\in\mathbb{Z}\}
$$
is said to be the {\it $1$-rank lattice}. We denote it by $\Gamma(a,b,N)$.
\end{definition}


All local minima (except the ones on the coordinate axes) are contained in the cube
$[-N/2,N/2]$ and, therefore, they form a finite set. In fact, a stronger statement holds.

\begin{proposition}\label{GeneralLattice}
Let $a$, $b$, and $N$ be arbitrary positive integers such that both $a$ and $b$ are relatively prime with $N$.
Then the set $\vrm(|\Gamma(a,b,N)|)$ is a finite axial set in general position.
\qed
\end{proposition}

As a consequence the Minkowski-Voronoi complex $MV(|\Gamma(a,b,N)|)$ is defined for all triples $(a,b,N)$ with
$a$ and $b$ being relatively prime with $N$.


\subsection{Geometric code for pairs and triples of integers}\label{geo_code}

\subsubsection{Pairs of integers}
Suppose $(a,N)$ be a pair of integers satisfying $a<N$.
As we have already mentioned in the introduction the lengths of ordinary continued fractions for
$$
\frac{N}{a},
\quad
\frac{N+a}{a},
\quad
\frac{N+2a}{a},
\quad
\frac{N+3a}{a},
\quad
\ldots
$$
coincide.

\vspace{2mm}

In order to approach the three-dimensional case we reformulate this statement as follows.

For a pair of integers $(a,N)$ consider the following three integers:
$$
\alpha= N \modd a, \quad a, \quad t=\lfloor N/a\rfloor.
$$
We say that $(\alpha, a, t)$ is the {\it geometric code} of $(a,N)$, where
the pair $(\alpha, a)$ is its {\it combinatorial part} and $t$ is a {\it parameter}.

In the new settings we have:
\\
{
\it The lengths of continued fractions for $(\alpha,a,t)$ for a fixed $(\alpha, a)$  and all $t\ge 1$ are the same.
}

\vspace{2mm}

\subsubsection{Triples of integers}
It is interesting to observe that there is a natural extension of the geometric code to the case of triples of integers.

\begin{definition}\label{Definition_code}
Let $(a,b,N)$ be a triple of positive integers, where $b\ge 2$ and $N$ is not divisible by  $b$.
The {\it geometric code} for a triple of nonnegative integers $(a,b,N)$ is defined as $(\alpha,\beta,\gamma,a,t,u)$ where
$$
\begin{array}{l}
\alpha = N \modd b,\\
\beta = b \modd (\alpha a),\\
\displaystyle
\gamma =\lfloor N/b \rfloor \modd a,
\end{array}
\qquad
\begin{array}{l}
\hbox{ }\\
\displaystyle
t= \big\lfloor b/(\alpha a)\big \rfloor,\\
\displaystyle
u=\big\lfloor \lfloor N/b\rfloor /a \big\rfloor.
\end{array}
$$
We will consider $(\alpha,\beta,\gamma,a)$ as {\it combinatorial part} and $(t,u)$ as {\it parametric part}.
\end{definition}

Notice that the natural bounds for the entries are:
\begin{equation}\label{GeometricCodeConditions}
\begin{array}{c}
0< \alpha < b;\\
0\le \beta < \alpha a\\
0\le \gamma <a;\\
t,u \ge 0.
\end{array}
\end{equation}

Additional conditions
\begin{equation}\label{GeometricCodeConditions2}
\begin{array}{l}
\gcd(\alpha,\beta)=1,
\\
\gcd(a,\alpha+\beta\gamma)=1
\end{array}
\end{equation}
are fulfilled if and only if $N$ is relatively prime with $a$ and $b$
in the corresponding triple $(a,b,N)$.

\begin{proposition}
There is a one-to-one correspondence between the set of triples $(a,b,N)$ of positive integers where $b\ge 2$ and $N$ is not divisible by  $b$,
and the set of all 6-tuples $(\alpha,\beta,\gamma,a,t,u)$ satisfying~$($\ref{GeometricCodeConditions}$)$.
\end{proposition}

\begin{proof}
The inverse map is given by:
$$
\begin{array}{l}
a=a\\
b=\alpha at+ \beta;\\
N= (\alpha at+\beta)(a u+\gamma)+\alpha.\\
\end{array}
$$
\end{proof}

\begin{corollary}\label{splitting}
The set of triples $(a,b,N)$ where $b\ge 2$ and $N$ is not divisible by $b$
is splitted into two-parametric families of triples by combinatorial types.
\end{corollary}

We use such two-parametric families in the next subsection.

\subsection{Main result and its proof outline}\label{MainResultFormulation}

Now we are ready to formulate the Minkowski-Voronoi complex stabilization theorem.


\begin{theorem}\label{stabilization}{\bf (Minkowski-Voronoi complex stabilization.)}
Consider a combinatorial part $(\alpha,\beta, \gamma, a)$
satisfying both conditions~$($\ref{GeometricCodeConditions}$)$
and~$($\ref{GeometricCodeConditions2}$)$. Let $u$ and $t$ be nonnegative
integer parameters. Set
$$
\begin{array}{l}
b(t)=\alpha at+ \beta;\\
N(t,u)=b(t) (a u+\gamma)+\alpha = (\alpha at+\beta)(a u+\gamma)+\alpha.\\
\end{array}
$$
Then the following statements hold.

{
\noindent
{\bf $(t,u)$-stabilization:}
there exist $t_0$ and $u_0$ such that for any $t>t_0$ and $u>u_0$ it holds
$$
MV\Big(\Gamma\big(a,b(t),N(t,u)\big)\Big)\approx MV\Big(\Gamma\big(a,b(t_0),N(t_0,u_0)\big)\Big).
$$
$($By $\approx$ we denote combinatorial equivalence relation for two complexes.$)$
}

{
\noindent
{\bf $t$-stabilization:}
for every $u\ge 0$ there exists $t_0$ such that for every $t>t_0$ we have
$$
MV\Big(\Gamma\big(a,b(t),N(t,u)\big)\Big)\approx MV\Big(\Gamma\big(a,b(t_0),N(t_0,u)\big)\Big).
$$
}

{
\noindent
{\bf $u$-stabilization:}
for every $t\ge 0$ there exists $u_0$ such that for every $u>u_0$ we have
$$
MV\Big(\Gamma\big(a,b(t),N(t,u)\big)\Big)\approx MV\Big(\Gamma\big(a,b(t),N(t,u_0)\big)\Big).
$$
}
\end{theorem}

\begin{remark}
Notice that conditions~(\ref{GeometricCodeConditions}) and~(\ref{GeometricCodeConditions2})
are equivalent to the fact that $a$ and $b$ are relatively prime with $N$.
Hence the corresponding rank-1 lattice $\Gamma(a,b,N)$ satisfies Proposition~\ref{GeneralLattice}.
\end{remark}

\begin{remark}\label{finitely_many}
Theorem~\ref{stabilization} implies that the number of distinct combinatorial types
of the Minkowski-Voronoi complexes for a given combinatorial part $(\alpha,\beta,\gamma,a)$ satisfying~(\ref{GeometricCodeConditions})
and~$($\ref{GeometricCodeConditions2}$)$ is finite.
\end{remark}

\begin{example}\label{m-ex}
Before to start the proof we study two examples. Consider
$$
\begin{array}{c}
1)\quad  a=2, \alpha=7, \beta=2, \gamma=0;\\
2)\quad  a=2, \alpha=7, \beta=2, \gamma=1.\\
\end{array}
$$
The corresponding Minkowski-Voronoi complexes are respectively as follows:
$$
\begin{array}{c}
\begin{array}{|c|c|c|}
\hline
&t=0&t>0\\
\hline
u=0&
\begin{array}{c}\includegraphics{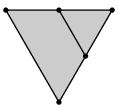}\end{array}&
\begin{array}{c}\includegraphics{4graphs_v3-1}\end{array}\\
\hline
u>0&
\begin{array}{c}\includegraphics{4graphs_v3-1}\end{array}&
\begin{array}{c}\includegraphics{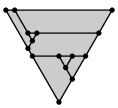}\end{array}\\
\hline
\end{array}
\\
(a=2, \alpha=7, \beta=2, \gamma=0)
\end{array}
\qquad
\begin{array}{c}
\begin{array}{|c|c|c|}
\hline
&t=0&t>0\\
\hline
u=0&
\begin{array}{c}\includegraphics{4graphs_v3-1}\end{array}&
\begin{array}{c}\includegraphics{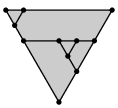}\end{array}\\
\hline
u>0&
\begin{array}{c}\includegraphics{4graphs_v3-1}\end{array}&
\begin{array}{c}\includegraphics{4graphs_v3-3}\end{array}\\
\hline
\end{array}
\\
(a=2, \alpha=7, \beta=2, \gamma=1)
\end{array}
$$
As you might notice, the resulting complexes are mainly combinatorially equivalent in these two families.
They are different only in the case $u=0$ and $t>0$.
In fact, it is quite common that if two combinatorial parts have the same values for $(a,\alpha,\beta)$ are different
values for $\gamma$, then the corresponding Minkowski-Voronoi complexes are rather similar and in some cases they are combinatorially equivalent.
\end{example}

\begin{remark}
Later on we return several times to the case $(a=2, \alpha=7, \beta=2, \gamma=1)$ in order to illustrate for the techniques proposed below.
\end{remark}

{\noindent {\bf Example of a single Minkowski-Voronoi complex computation.}}
Let is consider one particular example in more details.

\begin{example}\label{ex-ref}
The set up is
$$
a = 2,
\quad
\alpha = 7,
\quad
\beta = 2,
\quad
\gamma= 1.
$$
We set $u = 1$ and $t = 0$. Then by the formula on page~\pageref{stabilization} in Theorem~\ref{stabilization},
$$
\begin{array}{lllllll}
b&=&\alpha at +\beta
&=& 7\cdot 2\cdot 0 + 2
&=& 2,\\
N&=& (\alpha at +\beta)(au+\gamma)+\alpha
&=& 2 \cdot (2 \cdot 1 + 1) + 7
&=&13.\\
\end{array}
$$
Thus
$$
(a,b,N) = (2,2,13).
$$
By Definition~\ref{1-rank}, we are interested in the lattice
$$
\begin{aligned}
\Gamma(2,2,13)&=\{
m_1(1,a,b) + m_2(0,N,0) + m_3(0, 0,N) | m_1,m_2,m_3 \in \z\}
\\
&= \{m_1(1, 2, 2) + m_2(0, 13, 0) + m_3(0, 0, 13)
| m_1,m_2,m_3 \in \z\}.
\end{aligned}
$$
We aim to find the local minima of this lattice. By a lattice point, we will mean a non-zero lattice point. It is claimed right after the statement of Definition~\ref{1-rank} that all the relative minima, save for those on the coordinate axes,
are in the cube
$$
[-N/2,N/2] \times [-N/2,N/2] \times [-N/2,N/2].
$$
Thus the relative minima, save for the three on the coordinate axes, will be
in the cube
$$
[-6, 6] \times [-6, 6] \times[-6, 6].
$$
We now start the process for finding the relative minima.
We have in this set the following three elements on the coordinate axes:
$$
\begin{array}{llr}
(13, 0, 0)&=&13\cdot(1, 2, 2)-2\cdot(0, 13, 0)-2\cdot(0, 0, 13),\\
(0, 13, 0)&=&0 \cdot (1, 2, 2) + 1 \cdot (0, 13, 0) + 0 \cdot (0, 0, 13),\\
(0, 0, 13)&=&0 \cdot (1, 2, 2) + 0 \cdot (0, 13, 0) + 1 \cdot (0, 0, 13).\\
\end{array}
$$
We label these points as
$$
\gamma_1 = (13, 0, 0),
\qquad
\gamma_2 = (0, 13, 0),
\qquad
\gamma_3 = (0, 0, 13).
$$
Let us find the other relative minima.

Any lattice point of the form $m_2(0, 13, 0)+m_3(0, 0, 13)$
cannot be a local minima,
as all such points will contain either $\gamma_2$ or $\gamma_3$.

By listing all $m_1(1, 2, 2) + m_2(0, 13, 0) + m_3(0, 0, 13)$ (i.e., triples $m_1,m_2,m_3$), we see
that we get two more linearly independent local minima, namely
$$
\begin{array}{ccc}
(1, 2, 2)&=&1 \cdot (1, 2, 2) + 0 \cdot (0, 13, 0) + 0 \cdot (0, 0, 13),
\\
(6,-1,-1)&=&6 \cdot (1, 2, 2) - (0, 13, 0) - (0, 0, 13).
\end{array}
$$
We label these points as
$$
\gamma_4 = (1, 2, 2)
\quad \hbox{and} \quad
\gamma_5 = (6,-1,-1).
$$
We know that the faces of the Minkowski-Voronoi complex will correspond to these five local minima. The edges will correspond to pairs that are minimal.

In order to simplify the next steps we use the following notation. Let $S$ be a discrete set; consider the parallelepiped $\Pi(S)$. Let $(x,y,z)$ be the vertex $\Pi(S)$ with nonnegative $x$, $y$ nd $z$ coordinates. Similar to Example~\ref{ExVoronoi3} we denote $\Pi(S)$ by $[x,y,z]$.

Let us list below the parallelepipeds defining minimal edges for all ten possible pairs:
$$
\begin{array}{ll}
\Pi(\{\gamma_1,\gamma_2\}) =  [13, 13, 0], \quad &
\Pi(\{\gamma_2,\gamma_4\}) =  [1, 13, 2], \\
\Pi(\{\gamma_1, \gamma_3\}) = [13, 0, 13],\quad &
\Pi(\{\gamma_2,\gamma_5\}) =  [6, 13, 1],\\
\Pi(\{\gamma_1, \gamma_4\}) = [13, 2, 2],\quad  &
\Pi(\{\gamma_3,\gamma_4\}) =  [1, 2, 13],\\
\Pi(\{\gamma_1, \gamma_5\}) = [13, 1, 1], \quad &
\Pi(\{\gamma_3,\gamma_5\}) =  [6, 1, 13],\\
\Pi(\{\gamma_2,\gamma_3\}) =  [0, 13, 13],\quad &
\Pi(\{\gamma_4,\gamma_5\}) =  [6, 2, 2].
\end{array}
$$

We have
$$
\gamma_5\in
\Pi((\gamma_1,\gamma_4))
$$
meaning that
$(\gamma_1, \gamma_4)$
will not be an edge.
All the others are edges, and hence there will be nine edges.
The vertices will correspond to triple that are minimal. The ten possible triples are
$$
\begin{array}{ll}
\Pi(\{\gamma_1,\gamma_2,\gamma_3\}) = [13, 13, 13], \quad &
\Pi(\{\gamma_1,\gamma_4,\gamma_5\}) = [13, 2, 2],\\
\Pi(\{\gamma_1,\gamma_2,\gamma_4\}) = [13, 13, 2], \quad &
\Pi(\{\gamma_2,\gamma_3,\gamma_4\}) = [1, 13, 13],\\
\Pi(\{\gamma_1,\gamma_2,\gamma_5\}) = [13, 13, 1], \quad &
\Pi(\{\gamma_2,\gamma_3,\gamma_5\}) = [6, 13, 13],\\
\Pi(\{\gamma_1,\gamma_3,\gamma_4\}) = [13, 2, 13], \quad &
\Pi(\{\gamma_2,\gamma_4,\gamma_5\}) = [6, 13, 2],\\
\Pi(\{\gamma_1,\gamma_3,\gamma_5\}) = [13, 1, 13], \quad &
\Pi(\{\gamma_3,\gamma_4,\gamma_5\}) = [6, 2, 13].\\
\end{array}
$$
A triple will not be minimal if one of the other $\gamma_i$ not making up the triple is in its Voronoi minimal set. We see that
$$
\begin{array}{lll}
\gamma_4 = (1, 2, 2) &\in& \Pi(\gamma_1,\gamma_2,\gamma_3),\\
\gamma_5 = (6,-1,-1) &\in& \Pi(\gamma_1,\gamma_2,\gamma_4),\\
\gamma_5 = (6,-1,-1) &\in& \Pi(\gamma_1,\gamma_3,\gamma_4),\\
\gamma_4 = (1,2,2) &\in& \Pi(\gamma_2,\gamma_3,\gamma_5).\\
\end{array}
$$

In addition, for one of the vertices we have:
$\gamma_5=(6,-1,-1)$ is an interior point of $\Pi(\{\gamma_1,\gamma_4,\gamma_5\})$=[13,2,2].
Hence by Definition~\ref{def-minimal} (ii) the set $\{\gamma_1,\gamma_4,\gamma_5\}$ is not minimal.

Therefore, we end up with the following five vertices:
$$
\begin{array}{l}
v_1 = (\gamma_1,\gamma_2,\gamma_5),\\
v_2 = (\gamma_1,\gamma_3,\gamma_5),\\
v_3 = (\gamma_2,\gamma_3,\gamma_4),\\
v_4 = (\gamma_2,\gamma_4,\gamma_5),\\
v_5 = (\gamma_3,\gamma_4,\gamma_5).\\
\end{array}
$$

Finally, let us show the labeling for the edges of the diagram.
First consider the (oriented) edge $\{\gamma_1,\gamma_5\}$ starting at $v_1= (\gamma_1,\gamma_2,\gamma_5)$ and ending $v_2=(\gamma_1,\gamma_3,\gamma_5)$.
Here the parallelepiped $[13,13,1]$ changes to parallelepiped $[13,1,13]$. Hence its label is: $(y\downarrow , z\uparrow )$.

In the following table we give labels for the rest of edges.

\vspace{2mm}

\begin{center}
\begin{tabular}{|c|c|c|c|}
\hline
{\sc edge } & {\sc from} & {\sc \quad to\quad} & {\sc change}\\
\hline
$\{\gamma_1,\gamma_5\}$ & $v_1$ & $v_2$ & $(y\downarrow,z\uparrow)$\\
\hline
$\{\gamma_2,\gamma_5\}$ & $v_1$ & $v_4$ & $(x\downarrow,z\uparrow)$\\
\hline
$\{\gamma_3,\gamma_5\}$ & $v_2$ & $v_5$ & $(x\downarrow,y\uparrow)$\\
\hline
$\{\gamma_2,\gamma_4\}$ & $v_3$ & $v_4$ & $(x\uparrow,z\downarrow)$\\
\hline
$\{\gamma_3,\gamma_4\}$ & $v_3$ & $v_5$ & $(x\uparrow,y\downarrow)$\\
\hline
\end{tabular}
\qquad
\begin{tabular}{|c|c|c|c|}
\hline
{\sc edge } & {\sc from} & {\sc \quad to\quad } & {\sc change}\\
\hline
\hline
$\{\gamma_4,\gamma_5\}$ &$v_4$ & $v_5$ & $(y\downarrow,z\uparrow)$\\
\hline
$\{\gamma_1,\gamma_2\}$ &$v_1$ & $\infty$ & ---\\
\hline
$\{\gamma_1,\gamma_3\}$ &$v_2$ & $\infty$ & ---\\
\hline
$\{\gamma_2,\gamma_3\}$ &$v_3$ & $\infty$ & ---\\
\hline
\end{tabular}
\end{center}

\vspace{2mm}

Recall that the edges $\{\gamma_1,\gamma_2\}$, $\{\gamma_1,\gamma_3\}$, and $\{\gamma_2,\gamma_3\}$
are represented by rays in the Minkowski-Voronoi tessellation and hence they are not labeled.

Finally we show the Minkowski-Voronoi tessellation.
$$
\includegraphics{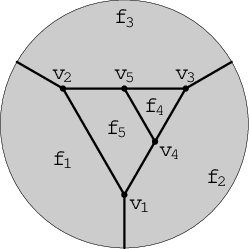}
$$
Here $f_i$ correspond to $\{\gamma_i\}$ for $i=1,\ldots, 5$.
\end{example}

\vspace{2mm}

{\noindent
{\bf General strategy to prove Theorem~\ref{stabilization}.}
We prove Minkowski-Voronoi complex stabilization Theorem in several steps.

First of all we study the structure of the set of relative minima
for the lattices $\Gamma(a,b,N)$ for triples with the same combinatorial part $(\alpha, \beta, \gamma, a)$
and with parameters $t$ and $u$.
Further for each $(a,b,N)$ we construct the list of relative minima $\Xi_{\alpha, \beta, \gamma, a}(t,u)$ satisfying nice properties:
\begin{itemize}
\item{The list $\Xi_{\alpha, \beta, \gamma, a}(t,u)$ contains all the relative minima of $\Gamma(a,b,N)$.}

\item{The number the elements in $\Xi_{\alpha, \beta, \gamma, a}(t,u)$ depends only on the combinatorial part, i.e., on $(\alpha, \beta, \gamma, a)$};

\item{The comparison relation between every $i$-th and $j$-th elements of different lists are the same
for sufficiently large $t$ (or respectively $u$, or $t$ and $u$).}
\end{itemize}
After the lists with such properties are constructed the proof is straightforward.
First, the Minkowski-Voronoi complexes for the lattices $\Gamma_N(a,b)$ coincide
with the Minkowski-Voronoi complexes for the points in the list $\Gamma(a,b,N)$.
Secondly, the combinatorial structures of the last ones are the same due
to the same comparison relations for sufficiently large parameters.

\vspace{2mm}

The remaining part of this section is organized as follows.
First we study the structure of the set of relative minima for the
lattices of a fixed family in Subsection~\ref{Classification of relative minima}.
Then in Subsection~\ref{Asymptotic comparison of coordinate functions} we formulate a notion of asymptotic comparison of
a pair of functions and prove the asymptotic comparison of the coordinate functions for the points of Minkowski-Voronoi complexes
in the family.
Finally, in Subsection~\ref{Conclusion of the proof} we conclude the proof of Theorem~\ref{stabilization}.
}

\subsection{Classification of relative minima of $\Gamma(a,b,N)$}\label{Classification of relative minima}

In this section we show that all relative minima of $\Gamma(a,b,N)$ fall into three categories of points,
whose coordinates possess regularities which we further use in the proof of Theorem~\ref{stabilization}.
The explicit description of these categories significantly reduces the construction time of the set of all relative minima for a given lattice.
Here we do not consider the case $t=0$ which is slightly different, we come back to it later in Subsection~\ref{Conclusion of the proof}.

\subsubsection{Some integer notation}
Let $a$ and $N$ be two numbers. Consider $0\le b<N$ such that $a-b$ is an integer divisible by $N$.
We denote
$$
b=a \modd N.
$$
Denote also
$$
|a|_N=\min(a \modd N, -a \modd N).
$$

\subsubsection{Three types of relative minima}
Consider the set $\vrm(|\Gamma(a,b,N)|)$ of all Voronoi relative minima.
Recall that each element of
$$
\vrm(|\Gamma(a,b,N)|)\setminus \big\{(N,0,0), (0,N,0), (0,0,N) \big\}
$$
is written in the form
$$
\big(x, |ax|_N, |bx|_N\big)
$$
for some integer $x \in \big[1,\frac{N}{2}\big]$.
In what follows we distinguish relative minima of three types defined by the first coordinate $x$.
First, let us decompose the interval
$I=\big[0,\frac{N}{2}\big)$ in the union
$$
I=\bigcup\limits_{k=0}^{a-1}I_k, \qquad \hbox{where} \qquad I_k=\left[\frac{k}{2a}N,\frac{k+1}{2a}N \right).
$$
Now every interval $I_k$ we decompose into three more intervals:
\begin{equation}
\label{parti}
I_k=I_{k,1}\bigcup I_{k,2}\bigcup I_{k,3},
\end{equation}
where for even $k$ we set
\begin{gather*}
  I_{k,1}=\left [\frac{k}{2a}N,\frac{k}{2a}N+1\right), \quad
  I_{k,2}=\left [\frac{k}{2a}N+1,\frac{k}{2a}N+au+\gamma\right), \\
  I_{k,3}=\left [\frac{k}{2a}N+au+\gamma,\frac{k+1}{2a}N\right),
\end{gather*}
and for odd $k$ we set
\begin{gather*}
  I_{k,1}=\left [\frac{k+1}{2a}N-1,\frac{k+1}{2a}N\right), \quad
  I_{k,2}=\left [\frac{k+1}{2a}N-(au+\gamma),\frac{k+1}{2a}N-1\right),  \\
  I_{k,3}=\left [\frac{k}{2a}N,\frac{k+1}{2a}N-(au+\gamma)\right).
\end{gather*}
We say that a relative minimum $\big(x, |ax|_N, |bx|_N\big)$ is {\it of the first type}, {\it of the second type},
or {\it of the third type} if
$x\in I_{k,1}$, $x\in I_{k,2}$, or $x\in I_{k,3}$ for some $k$ respectively.

There are three more minima of $\vrm(|\Gamma(a,b,N)|)$ except for the listed above, they are:
$(N,0,0)$, $(0,N,0)$, and $(0,0,N)$.
\vspace{2mm}

\begin{remark}
Notice that we have omitted the value $\frac{N}{2}$.
If $x=\frac{N}{2}$ then $|ax|_n,|bx|_n\in\{0,\frac{N}{2}\}$.
Both $|ax|_n$ and $|bx|_n$ are nonzero, since the set $\vrm(|\Gamma(a,b,N)|)$ is in general position.
Hence the only point to consider here is $(\frac{N}{2},\frac{N}{2},\frac{N}{2})$.
If this point is a relative minimum, then the point $(1,a,b)$ should not be in the parallelepiped
$\Pi\big(\frac{N}{2},\frac{N}{2},\frac{N}{2}\big)$, except if $(1,a,b)=\big(\frac{N}{2},\frac{N}{2},\frac{N}{2}\big)$.
This is the case only if $a{=}b{=}1,N{=}2$ which we do not study in Theorem~\ref{stabilization} ($b>1$ is not included by Definition~\ref{Definition_code}).
Therefore, the value $\frac{N}{2}$ can be omitted.
\end{remark}

\vspace{2mm}

{\noindent
{\bf Example~\ref{m-ex}, case $(a=2, \alpha=7, \beta=2, \gamma=1)$, continued,  part 1 of 4.}
Here we consider nonnegative integer parameters $t,u$. Recall that in our case
$$
b=b(t,u)=14t+2, \qquad N=N(t,u)=28tu+14t+4u+9.
$$
Since $a=2$, we have $2\cdot 3=6$ intervals in the decomposition of the unit segment $I$.
They are as follows:
$$
\begin{array}{ccc}
I_{0,1}=[0, 1),& I_{0,2}=[1, 2u{+}1), & I_{0,3}=\big[2u{+}1, N/4),\\
I_{1,3}=[N/4,N/2{-}2u{-}1), & I_{1,2}=[N/2{-}2u{-}1,N/2{-}1)&I_{1,1}=[N/2{-}1,N/2).\\
\end{array}
$$
}

\subsubsection{Properties of relative minima of different types}

Let us describe some important properties of the set of relative minima with respect to their types.

\begin{proposition}\label{StructuralProposition}{\bf Structural proposition.}
{\it $($i$)$} Let $(x,y,z)$ be a relative minimum of the first type of $\vrm(|\Gamma(a,b,N)|)$.
Then there exists a nonnegative integer $k\le a-1$ such that
\begin{align}\label{f1}
x=
\left\{
\begin{array}{cl}
\displaystyle\left\lceil\frac{k}{2a}N\right\rceil,     &\hbox{if $k$ is even},\\
\displaystyle\left\lfloor\frac{k+1}{2a}N\right\rfloor, &\hbox{otherwise}.
\end{array}
\right.
\end{align}

{\noindent
{\it $($ii$)$} Let $(x,y,z)$ be a relative minimum of the second type.
Then there exist a nonnegative even integer $k=2\hat k\le a-1$ and $\varepsilon\in \{0,1\}$ such that
$x$ equals to one of the following numbers
\begin{align}\label{f2}
  \left\lfloor N\left(\frac{\hat k}{a}-\frac{1}{b}\left\{\frac{\hat k\beta}{a}\right\}\right)\right\rfloor+\varepsilon, \quad
  \left\lfloor N\left(\frac{\hat k}{a}+\frac{1}{b}-\frac{1}{b}\left\{\frac{\hat k\beta}{a}\right\}\right)\right\rfloor+\varepsilon,
\end{align}
or there exist a nonnegative odd integer $k=2\hat k{+}1\le a-1$ and $\varepsilon\in \{0,1\}$ such that
$x$ equals to one of the following numbers
\begin{align}\label{f3}
  \left\lfloor N\left(\frac{\hat k+1}{a}-\frac{1}{b}\left\{\frac{(k+1)\beta}{a}\right\}\right)\right\rfloor+\varepsilon, \quad
  \left\lfloor N\left(\frac{\hat k+1}{a}+\frac{1}{b}-\frac{1}{b}\left\{\frac{(\hat k+1)\beta}{a}\right\}\right)\right\rfloor+\varepsilon,
\end{align}
}

{\noindent
{\it $($iii$)$} Let $(x,y,z)$ be a relative minimum of the third type.
Then it holds
\begin{align}\label{f4}
       |bx|_N\le \alpha.
\end{align}
}
\end{proposition}

\begin{proof}
We study two cases of odd and of even $k$ respectively.

\vspace{2mm}

{\noindent
{\bf The case of even $k$.} Let us consequently consider three types of relative minima.
}

\vspace{1mm}

{\noindent
{\it Type 1.} The interval $I_{k,1}$ is of unit length, and hence there is at most one local minimum with the first coordinate $x\in I_{k,1}$.
If such a minimum exists then $x$ satisfies~\eqref{f1}.
}

\vspace{2mm}

{\noindent
{\it Type 2.}
Let now $x=\frac{k}{2a}N+x_1\in I_{k,2}$. Then
$$
1<x_1<au+\gamma \qquad  \hbox{and} \qquad a<|ax|_N=ax_1<a(au+\gamma).
$$
Suppose that $\big(x, |ax|_N, |bx|_N\big)$ is a relative minimum.
Then the point $(1,a,b)$ should not be in the parallelepiped $\Pi\big(x, |ax|_N, |bx|_N\big)$,
which is possible only if
\begin{align}\label{bx}
       |bx|_N\le b.
\end{align}
From the definition of $b$ we know that the period of the function $f(x)=|bx|_N$  (as a function of real numbers) is exactly $N/b>au+\gamma$.
Therefore the interval $I_{k,2}$ contains at most two lattice points satisfying~(\ref{bx}).
These points are the endpoints of a unit interval containing a root of $f$.
The roots of $f$ that are the closest to the point $\frac{k}{2a}N=\frac{\hat k}{a}N$ will be
$$
  N\left(\frac{\hat k}{a}-\frac{1}{b}\left\{\frac{\hat k\beta}{a}\right\}\right), \qquad
  N\left(\frac{\hat k}{a}+\frac{1}{b}-\frac{1}{b}\left\{\frac{\hat k\beta}{a}\right\}\right).
$$
They lie at different sides from the point $\frac{\hat k}{a}N$.
(This follows from the equality $\big\{\frac{\hat kb}{a}\big\}=\big\{\frac{\hat k\beta}{a}\big\}$.)
Hence the first coordinate of the point $\big(x, |ax|_N, |bx|_N\big)$
of the second type should have one of the following values
\begin{align*}
  \left\lfloor N\left(\frac{\hat k}{a}-\frac{1}{b}\left\{\frac{\hat k\beta}{a}\right\}\right)\right\rfloor+\varepsilon, \quad
  \left\lfloor N\left(\frac{\hat k}{a}+\frac{1}{b}-\frac{1}{b}\left\{\frac{\hat k\beta}{a}\right\}\right)\right\rfloor+\varepsilon, \quad
  \varepsilon=0,1.
\end{align*}
This concludes the proof for the second type of relative minima that fall to the case of even $k$.
}

\vspace{2mm}

{
\noindent
{\it Type 3.}
Suppose now $x=\frac{k}{2a}N+x_1\in I_{k,3}$. Then
$$
au+\gamma<x_1<\frac{N}{2a}
\qquad \hbox{and} \qquad
a(au+\gamma)<|ax|_N=ax_1<\frac{N}{2}.
$$
In case if $\big(x, |ax|_N, |bx|_N\big)$ is a relative minimum,
the parallelepiped $\Pi\big(\big(x, |ax|_N, |bx|_N\big)\big)$ does not contain
the point
$$
(au+\gamma, |a(au+\gamma)|_N, |b(au+\gamma)|_N)=(au+\gamma, a(au+\gamma), \alpha),
$$
which holds only if $|bx|_N\le \alpha$.
}


\vspace{2mm}

{
\noindent {\bf The case of odd $k$.}
There exists at most one relative minimum with $x\in I_{k,1}$.
In case of existence, the first coordinate of the minimum is
\begin{align*}
       x=\left\lfloor\frac{k+1}{2a}N \right\rfloor.
\end{align*}
Similarly to the case of even $k$ the first coordinate of every relative minimum of the second types
equals to one of the coordinates of~\eqref{f3}, and the first coordinate of every relative minimum of the third type
satisfies inequality~\eqref{f4}.
The proofs here literally repeat the proofs for the case of even $k$, so we omit them.
}
\end{proof}

\subsubsection{Definition of the list $\Xi_{\alpha, \beta, \gamma, a}(t,u)$ and its basic properties}

Proposition~\ref{StructuralProposition} suggests the following definition.
\begin{definition}
For every $a$, $b=b(t)$, and $N=N(t,u)$ we consider the lattice $\Gamma(a,b,N)$.
Let us form a list $\Xi_{\alpha, \beta, \gamma, a}(t,u)$ of all points mentioned in Proposition~\ref{StructuralProposition}:
\vspace{1mm}
\begin{itemize}

\item First, we add to the list $a$ points of the first type, mentioned in Proposition~\ref{StructuralProposition}(i).
We enumerate them with respect of $k$.

\item Second, we consider $4a$ points of the second type of Proposition~\ref{StructuralProposition}(ii).
We enumerate them with respect to $k$, $\varepsilon$, and
the order in the strings~\eqref{f2} and~\eqref{f3}
(i.e., for the same $k$ and $\varepsilon$ we choose the left one before the right one).

\item
Third, we count $\alpha$ points of the third type mentioned in Proposition~\ref{StructuralProposition}(iii).
We choose the enumeration by the value of the last coordinate (i.e., by $|bx|_N$).

\item
Finally, we add three points $(N,0,0)$, $(0,N,0)$, and $(0,0,N)$.
\end{itemize}
\end{definition}

{
\noindent {\it Remark.}
Notice that some points in the list $\Xi_{\alpha, \beta, \gamma, a}(t,u)$ could be counted several times, we do this with intension to
use it further in the proof of Theorem~\ref{stabilization}.
}

\vspace{2mm}

{\noindent
{\bf Example~\ref{m-ex}, case $(a=2, \alpha=7, \beta=2, \gamma=1)$, continued, part 2 of 4.}
Here we have the following 20 points in the list.
$$
\begin{array}{ll}
\hbox{Type I:}  &p_{I1}=(0, 0, 0), \quad
                 p_{I2}=(14tu{+}7t{+}2u{+}4, 1, 7t{+}1),\\
\hbox{Type II:} &p_{II1}=(0, 0, 0), \quad
                 p_{II2}=(2u{+}1, 4u{+}2, 7), \\
&                p_{II3}=(1, 2, 14t{+}2),\quad
                 p_{II4}=(2u{+}2, 4u{+}4, 14t-5),\\
&                p_{II5}=(14tu{+}7t{+}2u{+}4, 1, 7t{+}1),\quad
                 p_{II6}=(14tu{+}7t{+}4, 4u + 1, 7t {+} 8),\\
&                p_{II7}=(14tu{+}7t{+}2u{+}4, 1, 7t{+}1),\quad
                 p_{II8}=(14tu{+}7t{+}3, 4y{+}3, 7t-6),\\
\hbox{Type III:}&p_{III1}=(12tu{+}6t{+}2u{+}4, 4tu{+}2t{+}1, 1),\quad
                 p_{III2}=(4tu{+}2t{+}1, 8tu{+} 4t{+} 2, 2),\\
&                p_{III3}=(8tu{+}4t{+}2u{+}3, 12tu{+}6t{+}3, 3),\\
&                p_{III4}=(8tu{+}4t{+}2, 12tu{+}6t{+}4u{+}5, 4),\\
&                p_{III5}=(4tu{+}2u+ 2t{+}2, 8tu{+}4t{+}4u{+}4, 5),\\
&                p_{III6}=(12tu{+}6t{+}3, 4tu{+}4u{+}2t{+}3, 6),\quad
                 p_{III7}=(2u{+}1, 4u{+}2, 7),\\
\hbox{Extra:}   &p_{01}=(28tu{+}14t{+}4u{+}9, 0, 0),\\
&                p_{02}=(0, 28tu{+}14t{+}4u{+}9, 0),\\
&                p_{03}=(0, 0, 28tu{+}14t{+}4u{+}9).
\end{array}
$$
As one can see we have some zero entry and some repeating entry in the list.
After removing them we have the following list of 15 vertices:
$$
\begin{array}{c}
p_{I2}, \quad
p_{II3}, \quad
p_{II4}, \quad
p_{II6}, \quad
p_{II8}, \quad
p_{III1}, \quad
p_{III2}, \quad
p_{III3}, \quad
p_{III4}, \quad
p_{III5}, \quad
p_{III6}, \quad
p_{III7}, \\
p_{01}, \quad
p_{02}, \quad
p_{03}.
\end{array}
$$
Note also, that for small $t$ and $u$ one should apply $|{*}|_N$ to every coordinate of every point.
}

\vspace{2mm}

From Proposition~\ref{StructuralProposition} we directly get the following corollary.
\begin{corollary}\label{ListVoronoi}
The following hold

{\noindent
{\it $($i$)$} Every relative minimum of $|\Gamma(a,b,N)|$ is contained in the list $\Xi_{\alpha, \beta, \gamma, a}(t,u)$.
}

{\noindent
{\it $($ii$)$} The set $\vrm(|\Gamma(a,b,N)|)$ contains at most $\alpha+5a+3$ elements.
}
\qed
\end{corollary}

\begin{remark}
The statements of this corollary give rise to a fast algorithm constructing relative Minkowski-Voronoi
complexes. Let us briefly outline the main stages of this algorithm.

{\it Stage 1:} Construct $\Xi_{\alpha, \beta, \gamma, a}(t,u)$.

{\it Stage 2:} Choose relative minima from the list $\Xi_{\alpha, \beta, \gamma, a}(t,u)$
(note that $(0,0,0)$ might happen in the list but it is not counted as a relative minimum).

{\it Stage 3:} Construct the diagram of the corresponding Minkowski-Voronoi complex.

{\noindent
Stages 1 and 2 are straightforward. Stage 3 is described above in Section~\ref{Canonical diagrams}.
The numbers of additions and multiplications used by this algorithm do not depend on $t$ and $u$.
}
\end{remark}

For a fixed combinatorial part $(\alpha, \beta, \gamma, a)$
let us consider the lists $\Xi_{\alpha, \beta, \gamma, a}(t,u)$ as a family with parameters $t$ and $u$.
Such family has the following remarkable properties.
First of all, all lists in the family have the same number of elements.
Secondly, the points of the lists with the same number
form a two-dimensional families whose properties are described in the next proposition.

\begin{proposition}\label{ListCoords}
For a fixed combinatorial part $(\alpha, \beta, \gamma, a)$ consider a family of lists $\Xi_{\alpha, \beta, \gamma, a}(t,u)$ with parameters $t$ and $u$.
Then for every $s\le \alpha+5a+3$ the $s$-th point in the lists $\Xi_{\alpha, \beta, \gamma, a}(t,u)$ is written as
$p_s(t,u)=(|x|_N,|y|_N,|z|_N)$ where
\begin{align}\label{f777}
     (x,y,z)=(A_1N+C_1u+D_1, A_2N+C_2u+D_2, A_3N+C_3t+D_3),
\end{align}
where the constants $A_i,B_i,C_i,D_i$ depend entirely on the combinatorial part $($i.e., on $a$, $\alpha$, $\beta$, and $\gamma$, and do not depend on
the parameters $t$ and $u$$)$.
\end{proposition}

{\noindent
{\bf Example~\ref{m-ex}, case $(a=2, \alpha=7, \beta=2, \gamma=1)$, continued, part 3 of 4.}
According to Proposition~\ref{ListCoords} we have
$$
\begin{array}{ll}
\hbox{Type I:} & p_{I2}=(\frac{1}{2}N-\frac{1}{2}, 1, 7t{+}1),\\
\hbox{Type II:} &p_{II3}=(1, 2, 14t{+}2),\quad
                 p_{II4}=(2u{+}2, 4u{+}4, 14t-5),\\
&                p_{II6}=(\frac{1}{2}N-2u-\frac{1}{2}, 4u + 1, 7t {+} 8),\\
&                p_{II8}=(\frac{1}{2}N-2u-\frac{3}{2}, 4y{+}3, 7t-6),\\
\hbox{Type III:}&p_{III1}=(\frac{3}{7}N+\frac{2}{7}u+\frac{1}{7}, \frac{1}{7}N-\frac{4}{7}u-\frac{2}{7}, 1)\\
&                p_{III2}=(\frac17N-\frac47u-\frac27,\frac27N-\frac87u-\frac47, 2),\\
&                p_{III3}=(\frac27N+\frac67u+\frac37,\frac37N-\frac{12}{7}u-\frac67, 3),\\
&                p_{III4}=(\frac27N-\frac87u-\frac47,\frac37N+\frac{16}{7}y+\frac87,4),\\
&                p_{III5}=(\frac17N+\frac{10}{7}u+\frac57,\frac27N+\frac{20}{7}u+\frac{10}{7}, 5),\\
&                p_{III6}=(\frac37N-\frac{12}{7}u-\frac67,\frac17N+\frac{24}{7}u+\frac{12}{7}, 6),\\
&                p_{III7}=(2u{+}1, 4u{+}2, 7),\\
\hbox{Extra:}   &p_{01}=(N, 0, 0),\\
&                p_{02}=(0, N, 0),\\
&                p_{03}=(0, 0, N).
\end{array}
$$
(Recall that we have deleted 5 points from the list: repeating points and zeroes.)
Here we list the points before finally applying $|{*}|_N$ to every coordinate of every point.
The application of $|{*}|_N$ will affect some values of the coordinates for small $t$ and for small $u$.
For instance, if $t=0$ then $|14t-5|_N=|-5|_9=4$, and it is not $-5$, while for $t\ge1$ we have $|14t-5|_N=14t-5$.
}

\vspace{2mm}

We start the proof with the following lemma.

\begin{lemma}\label{ListCoordsLemma}
Consider a family of points with coordinates $(x(t,u),ax(t,u),bx(t,u))$.
Suppose that there exist an integer $A$ and rational numbers $C$ and $D$
such that for every $t$ and $u$ the first coordinate of the family satisfies
\begin{align}\label{f7}
   x(t,u)=\frac{A}{a}N+Cu+D,\ \text{for some}\ A\in\z.
\end{align}
Then the family satisfies condition~$($\ref{f777}$)$.
\end{lemma}

\begin{proof}
The first coordinate satisfies condition~(\ref{f777}) by definition.
For the second and the third coordinates we have the following expressions:
\begin{align*}
   y&\equiv ax\equiv aCu+aD \quad \modd N;  \\
   z&\equiv bx\equiv \frac{A'}{a}N\beta+C(N-\alpha)+Db(t)\quad \modd N.
\end{align*}
Hence they also satisfy  condition~(\ref{f777}).
\end{proof}

Let us also recall the following definition.
\begin{definition}
{\it Continuants} $K_s$ ($s=0,1,\ldots$) are the polynomials that are defined iteratively as follows:
$$
\begin{array}{l}
    K_0()=1,\\
    K_1(x_1)=x_1,\\
    K_s(x_1,\ldots,x_n)=K_{n-1}(x_1,Е,x_{n-1})x_n+K_{n-2}(x_1,\ldots,x_{n-2}) \quad \hbox{for $n\ge 2$}.
\end{array}
$$
\end{definition}

{\noindent
{\it Proof of Proposition~\ref{ListCoords}.}
Let us fix some admissible $s_0$ and consider all the $s_0$-th entries in the lists $\Xi_{\alpha, \beta, \gamma, a}(t,u)$.
We study the points of three different types separately.
}

\vspace{1mm}

{
\noindent
{\it Points of the first type.} The first coordinates of the points of the first type are given by~\eqref{f1}.
Hence, equality~\eqref{f7} follows directly from the fact that
$
N\equiv(\beta\gamma+\alpha) \quad \modd a
$
and
$$
  \left\lfloor\frac{kN}{a}\right\rfloor=\frac{kN}{a}-\left\{\frac{k(\beta\gamma+\alpha)}{a}\right\}.
$$
}
Hence by Lemma~\ref{ListCoordsLemma} the points of the first type satisfy condition~\eqref{f777}.

{
\noindent
{\it Points of the second type.}
Consider now the points of the second type with even $k$ described by~\eqref{f2} (the case of odd $k$ is similar).
Notice that $\varepsilon$ contributes only to the constant $D$, so it is sufficient to study the case $\varepsilon=0$.
Since
$$
N\equiv\beta\gamma+\alpha \modd a \qquad  \hbox{and} \qquad \left\{\frac{k\beta}{a}\right\}au\in\z,
$$
we have
\begin{align*}
  \left[\frac{kN}{a}-\frac{N}{b}\left\{\frac{k\beta}{a}\right\}\right]&=
  k\frac{N-\beta\gamma-\alpha}{a}-\left\{\frac{k\beta}{a}\right\}au+
  \left[\frac{c(k)}{a}-\left\{\frac{k\beta}{a}\right\}\frac{\alpha}{b(t)}\right]       \\
 &=k\frac{N-\beta\gamma-\alpha}{a}-\left\{\frac{k\beta}{a}\right\}au+\left[\frac{c(k)-1}{a}\right],
\end{align*}
where $c(k)=(\beta\gamma+\alpha)k-a\gamma\left\{\frac{k\beta}{a}\right\}$ is an integer.
The last equality follows from the estimate $\frac{\alpha}{b(t)}<\frac{1}{a}$ for all $t\ge 1$.
}

So the first coordinates of the points of the second type satisfy equality~\eqref{f7}.
Hence by Lemma~\ref{ListCoordsLemma} the points of the first type satisfy the condition~\eqref{f777}.

\vspace{2mm}

{
\noindent
{\it Points of the third type.} Every point of the third type has the coordinates
$$
  (b'k,ab'k,k) \qquad \hbox{for $k\in\{1,2,\ldots,\alpha\}$,}
$$
where $b'$ satisfies $|bb'|_N=1$.
}

Let us find $b'$ explicitly. Consider the regular continued fraction expansion
$$
\frac{N}{b}=[a_0,\ldots, a_s].
$$
From the definition of $N$ and $b$ it follows that
$$
a_0= a u+\gamma, \quad
a_1=at+\Big\lfloor \frac{\beta}{\alpha}\Big\rfloor, \quad
\frac{\beta}{\alpha}-\Big\lfloor \frac{\beta}{\alpha}\Big\rfloor=[0,a_2,\ldots, a_s].
$$
Recall that
$$
\left(
\begin{array}{cc}
K_{s-1}(a_1,\ldots,a_{s-1})&K_{s}(a_1,\ldots,a_{s})\\
K_{s}(a_0,\ldots,a_{s-1})&K_{s+1}(a_0,\ldots,a_{s})\\
\end{array}
\right)
=
\left(
\begin{array}{cc}
K_{s-1}(a_1,\ldots,a_{s-1})&b\\
K_{s}(a_0,\ldots,a_{s-1})&N\\
\end{array}
\right)
$$
(here by $K_m$ we denote the corresponding continuant of degree $m$).
From general theory of continuants it follows that the determinant of the above matrix is $(-1)^s$, we have that
$$
|bK_{s}(a_0,\ldots,a_{s-1})|_N =1.
$$
Hence, without loss of generality we set
$$
b'=K_{s}(a_0,\ldots,a_{s-1}).
$$

Finally, let us examine the obtained expression for $b'$:
\begin{align*}
b'&=K_s(a_0,\ldots,a_{s-1})=a_0K_{s-1}(a_1,\ldots,a_{s-1})+K_{s-2}(a_2,\ldots,a_{s-1})\\
  &=(au+\gamma)(\zeta_1t+\zeta_2)+\zeta_3=\nu_1N+\nu_2u+\nu_3.
\end{align*}
where $\xi_i$ and $\nu_i$ ($i=1,2,3$) are constants.
Therefore, the first and the second coordinates of the points of the third type, which are equal to $kb'$ and $akb'$ respectively,
satisfy the conditions of the proposition. Finally, the third coordinates are constants equivalent to $k$ ($k=1,\ldots, \alpha$), and they satisfy
the conditions as well.
Therefore, the points of the third type satisfy the conditions of the proposition.
This concludes the proof.
\qed

\subsection{Asymptotic comparison of coordinate functions}\label{Asymptotic comparison of coordinate functions}

In this section we formulate a notion of asymptotic comparison and prove
two general statements that we will further use in the proof of Theorem~\ref{Conclusion of the proof}.

\begin{definition}
We say that a function $L:\z_+^2\to \r$
is {\it asymptotically stable}
if there exist numbers $\widehat{t}$ and $\widehat{u}$, such that the following
conditions hold:

\begin{itemize}
\item {\it $(t,u)$-stability condition}: for every $t\ge\widehat{t},\ u\ge\widehat{u}$ it holds
$L(t,u)=L(\widehat{t},\widehat{u})$;

\item {\it $u$-stability condition}: for every $t_0< \widehat{t}$ and every $u\ge\widehat{u}$ it holds
$L(t_0,u)=L(t_0,\widehat{u})$;

\item {\it $t$-stability condition}: for every $u_0< \widehat{u}$ and every $t\ge\widehat{t}$ it holds
$L(t,u_0)=L(\widehat{t},u_0)$.
\end{itemize}
\end{definition}

\begin{definition}
Two functions $F_1$ and $F_2$ are called {\it asymptotically comparable} if
the function
$\sign(F_1{-}F_2)$ is asymptotically stable.
\end{definition}

Let us continue with the following general statement.

\begin{proposition}\label{ModuleReduction}
Let $A$, $B$, $D$  be arbitrary integer numbers. Set
$$
F(u,t)=AN+Bt+D.
$$
Then the following statements hold.

\vspace{1mm}

{\noindent
{\it$($i$)$} There exist real numbers $A'$, $B'$, $D'$, and $\widehat{u}$ such that
for every $t\ge 0$ and $u>\widehat{u}$ we have
$$
|F(u,t)|_N=A'N+B't+D'.
$$
}

\vspace{1mm}

{\noindent
{\it$($ii$)$} For every $u_0\ge 0$ there exist real numbers $B''$, $D''$, and $\widehat{t}$ such that
for every $t>\widehat{t}$ we have
$$
|F(u,t)|_N=B''t+D''.
$$
}
\end{proposition}

\begin{remark}
It is clear that similar statements hold for the functions of type
$$
F(u,t)=AN+Cu+D
$$
(one should swap $t$ and $u$ in the conditions).
The proof in these settings repeats the proof of Proposition~\ref{ModuleReduction}, so we omit it.
\end{remark}

{
\noindent
{\it Proof of Proposition~\ref{ModuleReduction} $($i$)$.}
First if $|A|<1/2$, then there exists $\widehat{u}$ such that for every $u>\widehat{u}$ we have
$|F(t,u)|<N/2$. Therefore, for every $u>\widehat{u}$ we get
$$
|F(t,u)|_N=F(t,u).
$$
}

Secondly, let $A=1/2$. If $B< 0$ or $B=0$ and $D\le 0$ then there exists $\widehat{u}$ such that
for every $u>\widehat{u}$ we have $0<F(t,u)\le N/2$ and hence
$$
|F(t,u)|_N=F(t,u).
$$
If $B>0$ or $B=0$ and $D\ge 0$ then there exists $\widehat{u}$ such that
for every $u>\widehat{u}$ we have $N/2\le F(t,u)\le N$ and hence
$$
|F(t,u)|_N=N-F(t,u).
$$

Finally, the cases when $A\le -1/2$ or $A>1/2$ are reduced to the above two cases by adding or subtracting the number $N$ several times.
Then the statement follows directly from the fact that $|F(t,u)\pm N|_N=|F(t,u)|_N$.
\qed

\vspace{2mm}

{
\noindent
{\it Proof of Proposition~\ref{ModuleReduction} $($ii$)$.}
Let us fix $u_0\ge 1$. In this case we consider the function $F$ as a function in one variable $t$.
We write
$$
F(t,u_0)=Pt+Q,
$$
for some real numbers $P$ and $Q$. Let also
$$
N=N_1(u_0)t+N_2(u_0).
$$
}

If $|P|<N_1(u_0)/2$ then there exists $\widehat{t}$ such that for every $t>\widehat{t}$ we have
$$
|F(t,u_0)|_N=F(t,u_0).
$$
Consider now the case $P=N_1(u_0)/2$. If $Q\le N_2(u_0)/2$ then there exists $\widehat{t}$ such that for every $t>\widehat{t}$ we have
$0<F(t,u_0)\le N/2$ and hence
$$
|F(t,u_0)|_N=Pt+Q.
$$
If $Q> N_2(u_0)/2$ then there exists $\widehat{t}$ such that for every $t>\widehat{t}$ we have
$N/2<F(t,u_0)\le N$ and hence
$$
|F(t,u_0)|_N=N-Pt-Q.
$$

Finally, the cases $P\le -N_1(u_0)/2$ or $P> -N_1(u_0)/2$ are reduced to the above cases by adding or subtracting
the number $N$ several times.
Then the statement follows directly from the fact that $|F(t,u_0)\pm N|_N=|F(t,u_0)|_N$.
\qed

\vspace{2mm}

In order to compare the coordinates of the points in the lists $\Xi_{\alpha, \beta, \gamma, a}(t,u)$ we
formulate and prove the following statement.

\begin{proposition}\label{ModComparison}
Let $F_1$ and $F_2$ be a pair of functions of two variables as in one of the following cases:

\vspace{1mm}
{\noindent
\quad $($i$)$ \quad $F_1(t,u)=A_1N+B_1t+D_1, \qquad F_2(t,u)=A_2N+B_2t+D_2$;
}

{\noindent
\quad $($ii$)$ \quad $F_1(t,u)=A_1N+C_1u+D_1, \qquad F_2(t,u)=A_2N+C_2u+D_2$.
}
\vspace{1mm}

{\noindent
Then the functions $|F_1|_N$ and $|F_2|_N$ are asymptotically comparable.
}
\end{proposition}

\begin{example}
Let us show a comparison for the coordinates of Example~\ref{m-ex} (we have obtained them in part 2 of 4).
We compare the expressions for $x$ coordinates of $p_{III2}$ and $p_{III7}$, which are as follows:
$$
x_1=|4tu+2t+1|_N\quad  \hbox{and}\quad x_2=|2u+1|_N,
$$
where $N=28tu+14t+4u+9$.

We have the following three distinct cases here

\vspace{1mm}

\begin{itemize}
\item If $t=u=0$, then $x_1-x_2=0$.

\item If $t=0$ and $u>0$, then $x_1-x_2=-2u<0$.

\item If $t\ge 1$ and $u\ge 1$ then $x_1-x_2=(4t-2)u+2t>0$.
\end{itemize}

\vspace{1mm}

{
\noindent
As we see the expressions for $x_1$ and $x_2$ are asymptotically comparable.
All the other comparisons of Example~\ref{m-ex} are similar and we skip them here.
}
\end{example}

Without loss of generality we restrict ourselves to the first item (the proof for the second item repeats the proof for the first one).
We start the proof with the following lemma.

\begin{lemma}\label{ModComparisonLemma}
The functions $F_1$ and $F_2$ are asymptotically comparable.
\end{lemma}

\begin{proof}
By the definition it is sufficient to show that the function $F=F_1-F_2$ is comparable with the zero function.
So let
$$
F(t,u)=AN+Bt+D.
$$

If $A=0$ then for $t>-D/B$ the function $F$ does not change its sign, for $t<-D/B$ the function $F$ does not change sign,
and for any $u$ we have $F(-D/B,u)=0$. Hence $F$ is asymptotical comparable with the zero function.

If $A\neq 0$ then the equation $F(t,u)=0$ defines a hyperbola on $(t,u)$-plane, whose asymptotes are parallel to coordinate axes.
Hence we have the asymptotic comparison of $F$ and the zero function directly from definition
(notice that we essentially use the fact that the function is defined over $\z_+^2$).
\end{proof}

{
\noindent {\it Proof of Proposition~\ref{ModComparison}.}
{\it Verification of $(t,u)$-stability and $u$-stability.} We show $(t,u)$-stability and $u$-stability condition
of the function $\sign(|F_1|_N-|F_2|_N)$ simultaneously.
By Proposition~\ref{ModuleReduction} $($i$)$ we know that there exists $\widehat{u}_0$ such that
for every $t\ge 1$ and $u>\widehat{u}_0$ both functions $|F_1|_N$ and $|F_2|_N$ are written as
$$
|F_1(u,t)|_N=A'_1N+B'_1t+D'_1 \quad \hbox{and} \quad |F_2(u,t)|_N=A'_2N+B'_2t+D'_2.
$$
By Lemma~\ref{ModComparisonLemma}, such functions are asymptotically comparable.
Hence there exist some constants $\widehat{u}_1$ and $\widehat{t}_1$ satisfying simultaneously
$(t,u)$-stability and $u$-stability condition.
}

\vspace{2mm}

{
\noindent
{\it Verification of $t$-stability.}
Similarly from Proposition~\ref{ModuleReduction} $($ii$)$ follows the existence of
the constant $\widehat{t}_2$ satisfying $t$-stability condition for every $u_0<\widehat{u}_1$.
}

\vspace{1mm}

Finally, the constants $\widehat{t}=\max(\widehat{t}_1,\widehat{t}_2)$ and $\widehat{u}=\widehat{u}_1$
satisfy all three stability conditions. Therefore, $F_1$ and $F_2$ are asymptotically stable.
This concludes the proof of Proposition~\ref{ModComparison} {\it$($i$)$}.

\vspace{2mm}

As we have already mentioned the proof of Proposition~\ref{ModComparison} {\it$($ii$)$} is similar and we omit it.
\qed

\subsection{Conclusion of the proof of Theorem~\ref{stabilization}}\label{Conclusion of the proof}

First of all we fix the combinatorial type $(\alpha,\beta,\gamma,a)$. Suppose that $t\ge 1$.
By Corollary~\ref{ListVoronoi}{\it$($i$)$} every relative minimum of $|\Gamma(a,b,N)|$ is contained in the list $\Xi_{\alpha, \beta, \gamma, a}(t,u)$.
Further by Proposition~\ref{ListCoords}
for every admissible $s$ the $s$-th entries $p_s(t,u)$ in the lists $\Xi_{\alpha, \beta, \gamma, a}(t,u)$ are written as
$$
      p_s(t,u)=\big(\big|A_1N+C_1u+D_1\big|_N, \big|A_2N+C_2u+D_2\big|_N, \big|A_3N+C_3t+D_3\big|_N\big), \quad t\ge 1, u\ge 1.
$$
Hence by Proposition~\ref{ModComparison} for every admissible $s_1$ and $s_2$ each coordinate of
$p_{s_1}(t,u)$ is asymptotically comparable (with respect to $t$ and $u$) with the corresponding
coordinate of $p_{s_2}(t,u)$.
In other words, starting from some positive integers we have $(t,u)$-, $t$-, and $u$-stabilizations of all inequalities for all the coordinates
for all the corresponding pairs of points in the lists (it is important that the number of points in every list is exactly $\alpha+5a+3$).
Recall that every Minkowski-Voronoi complex is defined completely by the inequalities between the same coordinates of different points.
Therefore, the family of Minkowski-Voronoi complexes $MV\big(\Gamma\big(2,b(t),N(t,u)\big)$ is $(t,u)$-, $t$-, $u$-stable.

If $t=0$ then $b=\beta$ is a constant and
$N=\beta(au+\gamma)+\alpha$. In this case instead of partition~\eqref{parti}
we divide $I_k$ in a simpler way:
$$
I_k=I_{k,1}\bigcup I_{k,3},
$$
where for even $k$ we set
\begin{gather*}
  I_{k,1}=\left [\frac{k}{2a}N,\frac{k}{2a}N+1\right), \quad
  I_{k,3}=\left [\frac{k}{2a}N+1,\frac{k+1}{2a}N\right),
\end{gather*}
and for odd $k$ we set
\begin{gather*}
  I_{k,1}=\left [\frac{k+1}{2a}N-1,\frac{k+1}{2a}N\right), \quad
  I_{k,3}=\left [\frac{k}{2a}N,\frac{k+1}{2a}N-1\right).
\end{gather*}
On the intervals $I_{k,3}$ instead of~\eqref{f4} we shall have the inequality $ |bx|_N\le
\beta$. As before it means that $
|bx|_N$ on the interval $I_{k,3}$ is bounded by an absolute constant and the rest part of the proof remains the same.
\qed

\vspace{2mm}

{\noindent
{\bf Example~\ref{m-ex}, case $(a=2, \alpha=7, \beta=2, \gamma=1)$, continued, part 4 of 4.}
In our case we have the following diagrams.
\begin{center}
$\includegraphics{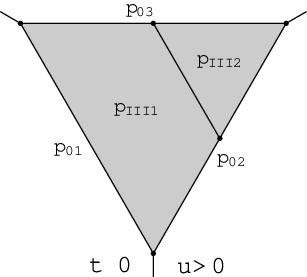}$
\quad
$\includegraphics{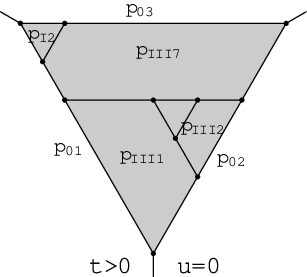}$
\quad
$\includegraphics{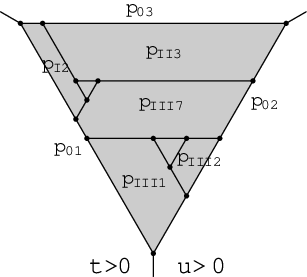}$
\end{center}
In the middle of each face we write the corresponding relative minimum.
Here $\circ=p_{III6}$ and $*=p_{II6}$. The construction is completed.
}


\section{A few words about lattices $\Gamma(a,b,N)$ with small $a$}

\subsection{Alphabetical description of canonical diagrams for lattices $\Gamma(a,b,N)$ with small $a$}

In this subsection we say a few words about lattices with a small parameter $a$.
Observed experiments suggest that canonical diagrams of such lattices are rather simple,
there is a good way to describe their combinatorics.
In order to do this we cut a diagram along the parallel lines  in the horizontal direction, as it is shown in the example below.

$$
\begin{array}{ccccccc}
\begin{array}{c}\includegraphics{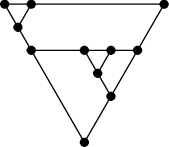}\end{array}
&\to &
\begin{array}{c}\includegraphics{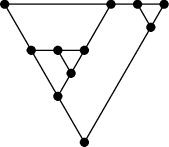}\end{array}
&\to &
\begin{array}{c}\includegraphics{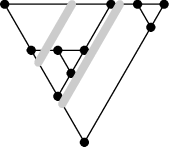}\end{array} &\to &
\begin{array}{c}\includegraphics{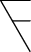} \hspace{1mm} \includegraphics{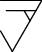} \hspace{1mm} \includegraphics{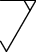}\end{array}
\end{array}
$$
In this example

--- first, we consider a canonical diagram for some $S$ (the first picture from the left);

--- then we rotate it by $\frac{2\pi}{3}$ clockwise (the second picture from the left);

--- further we cut it in several parts by parallel cuts (the third picture from the left);

--- finally, we redraw it in the symbolic form (the last picture from the left).

{\noindent
In some sense each diagram is written as a word in special letters,
encoding the combinatorics of the obtained pieces after performing cuts.
We choose the letters in the word to be similar to the corresponding parts
(after the rotation by the angle $\frac{2\pi}{3}$).
}

Let us discuss diagram decompositions in the simplest case of $a=1,2$.

\subsection{The case of lattices $\Gamma(1,b,N)$}\label{Case1bN}

We start with the case $a=1$. It turns out that in this case every canonical diagram is represented by a word consisting
of two letters ``$\includegraphics{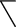}$'' and ``$\includegraphics{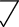}$''.

\begin{theorem}\label{White}
Let $b$ and $N$ be relatively prime positive integers, such that $b\le\frac{N}{2}$.
Then the canonical diagram of the set $|L(1,b,N)|$ is defined by the following word:
$$
\includegraphics{let_sm-1}\hspace{5.mm}
\includegraphics{let_sm-11} \hspace{5.mm}
\includegraphics{let_sm-11}\hspace{5.mm}
\ldots
\hspace{5.mm}\includegraphics{let_sm-11},
$$
where the number of letter ``$\hspace{5.mm}\includegraphics{let_sm-11}$'' equals to the number of elements in the shortest regular continued fractions of $\frac{N}{b}$.
\end{theorem}

\begin{remark}
The lattices of Theorem~\ref{White} have a remarkable property. They enumerate all lattice tetrahedra whose interiors do not contain lattice points.
This classification was considered for the first time by G.K.~White in~\cite{WhiteCite}.
\end{remark}

\begin{proof}
The proof is based on all Voronoi relative minima enumeration.

Consider an ordinary continued fraction for $b/N$:
$$
  \frac{b}{N}=[0;a_1,\ldots,a_s]\qquad (a_1\ge 2,\ a_s\ge 2).
$$
Then, as it is shown by G.F.~Voronoi~\cite{Voronoi1896,Voronoi1952} all relative minima of the two-dimensional lattice generated by the vectors $(N,0)$ and $(1,b)$
(we denote this lattice by $\Gamma(b,N)$)
are of the form
\begin{align*}
   \gamma_j=(x_j,y_j)=\left((-1)^{j+1}K_{j-1}(a_1,\ldots,a_{j-1}),K_{s-j}(a_{j+1},\ldots,a_s)\right)
   \quad (0\le j\le s+1),
\end{align*}
In particular, we have  $\gamma_0=(0,N)$, $\gamma_1=(1,b)$, $\gamma_{s+1}=(\pm N,0)$.

Set
$$
\tilde \gamma_j=(|x_j|,|x_j|,|y_j|), \quad \hbox{where $\gamma_j=(x_j,y_j)$}.
$$

Recall that the lattice $L(1,b,N)$ is generated by $(1,1,b)$, $(0,N,0)$, and $(0,0,N)$.
Notice that the set of Voronoi relative minima $\vrm(|L(1,b,N)|)$ contains both $(0,N,0)$ and $(0,0,N)$. All the other
Voronoi relative minima are the points of type
$$
(k \modd N,k \modd N, kb \modd N) \in [0,N]\times[0,N]\times[0,N].
$$
The first two coordinates of such points coincide with each other.
Therefore, it is a local minimum in $\Gamma(1,1,b)$
if and only if the point $(k\modd N,kb \modd N)$ is a local minimum in the lattice
$\Gamma(b,N)\subset\z^2$. Therefore, the set of all relative minima is as follows:
$$
\big\{
\gamma_x, \gamma_y, \gamma_z, \tilde \gamma_1,\ldots, \tilde \gamma_s
\big\},
$$
where $\gamma_x=(N,0,0)$, $\gamma_y=(0,N,0)$, and $\gamma_z=(0,0,N)$.
So the vertices of the Minkowski-Voronoi complex correspond to the following triples
$$
(\gamma_z,\tilde\gamma_1,\gamma_x), (\gamma_z,\tilde\gamma_1,\gamma_y),
(\tilde\gamma_1,\tilde\gamma_2,\gamma_x), (\tilde\gamma_1,\tilde\gamma_2,\gamma_y),
\ldots,
(\tilde\gamma_{s-1},\tilde\gamma_{s},\gamma_x), (\tilde\gamma_{s-1},\tilde\gamma_{s},\gamma_y),
(\tilde\gamma_{s},\tilde\gamma_x,\gamma_y).
$$
Direct calculations show that the corresponding diagrams are as stated in the theorem.
\end{proof}

\subsection{The case of lattices $\Gamma(2,b,N)$}

We conjecture that the alphabet for the case $a=2$ consists of 14 letters.

\begin{conjecture}\label{conjecture-3}
Let $b$ and $N$ be relatively prime positive integers, such that $b\le\frac{N}{2}$.
Then the canonical diagram of Minkowski-Voronoi complex for the lattice $L(2,b,N)$ is defined by the words
whose letters are contained in the following alphabet:

\begin{center}
\begin{tabular}{cccccccccccccc}
$$\includegraphics{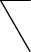}$$ & $$\includegraphics{letters-2}$$ &
$$\includegraphics{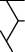}$$ & $$\includegraphics{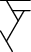}$$ &
$$\includegraphics{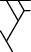}$$ & $$\includegraphics{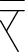}$$ &
$$\includegraphics{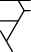}$$ & $$\includegraphics{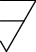}$$ &
$$\includegraphics{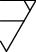}$$ & $$\includegraphics{letters-10}$$ &
$$\includegraphics{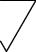}$$ & $$\includegraphics{letters-12}$$ &
$$\includegraphics{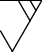}$$ & $$\includegraphics{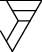}$$\\
$0$&$A$&$a$&$b$&$c$&$p$&$q$&$x$&$y$&$z$&$1$&$2$&$3$&$4$\\
\end{tabular}
\end{center}
\end{conjecture}

\begin{remark}
For simplicity we substitute the letters of this alphabet by the characters
$0$,  $A$,  $a$,  $b$, $c$, $p$,  $q$, $x$, $y$, $z$,  $1$--$4$ as above.
Letters $0$ and $A$ always take the first position.
The remaining part of the word splits in the blocks of two types.
A {\it simple block} is a block with a simple number $0$, $1$, $2$, $3$, or $4$.
A {\it nonsimple block} starts
with $A$,  $a$,  $b$, or $c$ it can have several letters $p$ and $q$ in the middle and ends with $x$, $y$, or $z$.
We separate such blocks with spaces.
So in some sense a word in the original alphabet is a sentence in characters
$0$,  $A$,  $a$,  $b$, $c$, $p$,  $q$, $x$, $y$, $z$,  $1$--$4$.
\end{remark}

\begin{remark}
The conjecture is checked for all lattices $L(2,b,N)$ with $b\le 11$ and some other particular examples.
\end{remark}

\begin{example}
Let us consider the lattice $\Gamma(2,26,121)$. The corresponding Minkowski-Voronoi complex
has the following canonical diagram:
$$\includegraphics{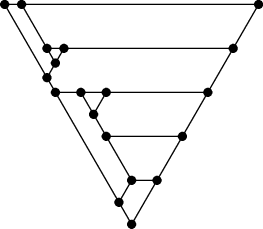}$$
The corresponding word is
$
\includegraphics{let_sm-1} \includegraphics{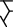} \includegraphics{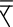}
\includegraphics{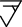} \includegraphics{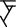} \includegraphics{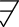}
$,
which is written in new characters as: $0$ $apz$ $bx$.
\end{example}

\begin{example}
Finally, let us list all stable configurations
that are described in Minkowski-Voronoi complex stabilization theorem
for $a=2$ and $\alpha\le 6$ (in the notation of Theorem~\ref{stabilization}).

\begin{center}
\begin{tabular}{|l|l|l|c|}
\hline
$\alpha=1$ & $\beta=1$, $\gamma=0$ & $u\ge 2$, $v\ge 2$ & $0$ $3$ $2$ \\
\hline
$\alpha=2$ & $\beta=1,3$; $\gamma=1$ & $u\ge 1$, $v\ge 1$ & $Az$ $2$\\
\hline
$\alpha=3$ & $\beta=1$; $\gamma=0$ & $u\ge 2$, $v\ge 2$ &  $0$ $2$ $3$ $2$\\
           & $\beta=2$; $\gamma=0,1$ & $u\ge 1$, $v\ge 1$ & $Ax$ $bx$\\
           & $\beta=4$; $\gamma=0,1$ & $u\ge 1$, $v\ge 1$ & $0$ $2$ $bx$\\
           & $\beta=5$; $\gamma=0$ & $u\ge 2$, $v\ge 1$ & $Ax$ $3$ $2$\\
\hline
$\alpha=4$ & $\beta=1,5$; $\gamma=1$ & $u\ge 1$, $v\ge 1$ & $0$ $bz$ $2$\\
           & $\beta=3,7$; $\gamma=1$ & $u\ge 1$, $v\ge 1$ & $0$ $apz$ $2$\\
\hline
$\alpha=5$ & $\beta=1$; $\gamma=0$ & $u\ge 2$, $v\ge 2$ & $0$ $3$ $3$ $2$\\
           & $\beta=2$; $\gamma=0,1$ & $u\ge 1$, $v\ge 1$ & $Az$ $bx$\\
           & $\beta=3$; $\gamma=0$ & $u\ge 2$, $v\ge 2$ & $Apy$ $3$ $2$\\
           & $\beta=4$; $\gamma=0,1$ & $u\ge 1$, $v\ge 1$ & $0$ $4$ $bx$\\
           & $\beta=6$; $\gamma=0,1$ & $u\ge 1$, $v\ge 1$ & $0$ $3$ $bx$\\
           & $\beta=7$; $\gamma=0$ & $u\ge 2$, $v\ge 1$ & $Az$ $3$ $2$\\
           & $\beta=8$; $\gamma=0,1$ & $u\ge 1$, $v\ge 1$ & $Apy$ $bx$\\
           & $\beta=9$; $\gamma=0$ & $u\ge 2$, $v\ge 1$ & $0$ $4$ $3$ $2$\\
\hline
$\alpha=6$ & $\beta=1,7$; $\gamma=1$ & $u\ge 1$, $v\ge 1$ & $0$ $bz$ $2$\\
           & $\beta=5,11$; $\gamma=1$ & $u\ge 1$, $v\ge 1$ & $0$ $cpz$ $2$\\
\hline
\end{tabular}
\end{center}
\end{example}

Experiments show that not all possible configurations of letters are realizable.
One of the natural questions here is as follows.

\begin{problem}
{\it $($i$)$} Find combinations of letters that are not realizable for $\Gamma(2,b,N)$ lattices.

{\noindent
{\it $($ii$)$} Find combinations of letters that are not realizable for rank-1 integer lattices.
}
\end{problem}

Finally we would like to raise the following general question for $a>2$.
\begin{problem}
Let $a>2$ be an integer. Does there exist a finite alphabet describing all the diagrams for $\Gamma(a,b,N)$?
\end{problem}
In fact, we have some evidences of the existence of finite alphabets for $a>2$, although the number of letters in them
might be relatively large.

\vspace{5mm}

{\noindent
{\bf Acknowledgement.} Some part of the work related to this paper was performed at TU Graz.
Both authors are grateful to TU Graz for hospitality and excellent working conditions.
We are grateful to the unknown reviewer for valuable comments and remarks,
in particular for providing us with an exhaustive example of a Minkowski-Voronoi complex computation (Example~\ref{ex-ref}).
Oleg Karpenkov is partially supported by EPSRC grant EP/N014499/1 (LCMH).
Alexey Ustinov is supported by Grant of the Government of Khabarovsk Krai (Order N 479-p, June 29, 2016).
}

\vspace{1cm}

\end{document}